\documentclass[11pt]{amsart}

\usepackage{graphicx}
\usepackage{amsmath}
\usepackage{amsfonts}
\usepackage{amssymb}
\usepackage{amsthm}
\usepackage{paralist}
\usepackage{bbm}
\usepackage{url}
\usepackage[matrix,arrow,cmtip,curve,dvips]{xy}

\usepackage{hyperref}
\usepackage[margin=1.2in]{geometry}

\usepackage[boxed,vlined,linesnumbered]{algorithm2e}
\SetKwInOut{Input}{input}
\SetKwInOut{Output}{output}
\SetArgSty{rm}
\SetFuncSty{tt}
\setalcapskip{1ex}
\renewcommand\thealgocf{\Alph{algocf}}

\renewcommand\P{\mathcal{P}}                % polyhedral complex
\newcommand\Q{\mathcal{Q}}                % polyhedral complex, other labels

\newcommand\x{{\mathbf{x}}} % monomial
\renewcommand\a{{\mathbf{a}}} % generic exponent
\newcommand\field{{\mathbbm{k}}} % Reals
\newcommand\R{{\mathbb{R}}} % Reals
\newcommand\C{{\mathbb{C}}} % Complex
\newcommand\N{{\mathbb{N}}} % Natural numbers
\newcommand\Z{{\mathbb{Z}}} % integers
\newcommand\ZO{\{0,1\}} % Natural numbers
\newcommand\puiseux{\C\{\!\{z\}\!\}} % Puiseux series with complex coefficients
\newcommand\generalizedpuiseux{{\C\{\!\{z\}\!\}^\text{gen}}} % generalized Puiseux series with complex coefficients
\newcommand\Sym{\operatorname{Sym}}

\renewcommand\a{\mathbf{a}}         % a in Z^n
\renewcommand\b{\mathbf{b}}         % b in Z^n

\newcommand\res{\mathcal{F}}        % free algebraic complex / free resolution
\newcommand\ideal[1]{\langle{#1}\rangle}
\newcommand\iform{\operatorname{in}}

\newcommand\ct{\mathbf{t}}      % coarse type
\newcommand\cct{\overline{\ct}}      % coarse type
\newcommand\dt{\mathbf{d}}      % dual coarse type
\newcommand\ft{T}               % fine type
\newcommand\fct{\overline{\ft}} % fine cotype

\newcommand\cI{I_{\ct(\A)}}             % coarse type ideal
\newcommand\ccI{I_{\overline{\ct}(\A)}} % coarse cotype ideal
\newcommand\fI{I_{\ft(\A)}}           % fine cotype ideal
\newcommand\fcI{I_{\fct(\A)}}           % fine cotype ideal

\newcommand\Rad{\operatorname{Rad}} % radical

\newcommand\dist{\operatorname{dist}} % distance
\newcommand\conv{\operatorname{conv}} % convex hull
\newcommand\tconv{\operatorname{tconv}} % tropicalconvex hull
\newcommand\trop{\operatorname{trop}} % tropicalization
\newcommand\vones{\mathbbm{1}} % all ones vector
\newcommand{\SetOf}[2]{\left\{#1\,:\,#2\right\}}
\newcommand{\smallSetOf}[2]{\{#1:#2\}}

\newcommand\T{\mathbb{T}} % tropical Torus
\newcommand\A{{\mathcal{A}}} % Arrangement
\newcommand\CD{\mathcal{C}} % cell decomposition induced by type
\newcommand\boundedCD{\mathcal{B}} % bounded part of cell decomposition induced by type
\newcommand\m{{\mathfrak{m}}} % maximal ideal
\newcommand\vol{\operatorname{Vol}} % volume

\newcommand\ccut[1]{\operatorname{CrossCut}(#1)} % cross cut complex
\newcommand\Cayley{\operatorname{Cayley}} % Cayley embedding

\newtheorem{thm}{Theorem}[section]
\newtheorem{cor}[thm]{Corollary}
\newtheorem{lem}[thm]{Lemma}
\newtheorem{prop}[thm]{Proposition}

\theoremstyle{definition}
\newtheorem{dfn}[thm]{Definition}
\newtheorem{example}[thm]{Example}
\newtheorem{rem}[thm]{Remark}

\begin{document}
%\setdefaultleftmargin{1em}{}{}{}{.5em}{.5em}
\setdefaultitem{$\triangleright$}{}{}{}

\title[Tropical types and associated cellular resolutions]%
{Tropical types and\\ associated cellular resolutions}

\author{Anton Dochtermann}
\address{Anton Dochtermann\\
         Department of Mathematics\\
         6188 Kemeny Hall\\
         Dartmouth College\\
         Hanover, NH 03755\\
         USA}
\email{anton.dochtermann@gmail.com}
\thanks{Research of Dochtermann supported by an Alexander von Humboldt postdoctoral fellowship.}

\author{Michael Joswig}
\address{Michael Joswig\\
         Fachbereich Mathematik\\
         TU Darmstadt\\
         Dolivostr. 15\\
         64293 Darmstadt\\
         Germany}
\email{joswig@mathematik.tu-darmstadt.de}
\thanks{Research of Joswig supported by Deutsche Forschungsgemeinschaft, DFG Research Unit ``Polyhedral Surfaces''.}

\author{Raman Sanyal}
\address{Raman Sanyal\\
	Department of Mathematics\\
        University of California\\
        Berkeley, California 94720 \\
        USA}
\email{sanyal@math.berkeley.edu}
%\urladdr{http://www.math.berkeley.edu/\~{}sanyal}
\thanks{Research of Sanyal supported by the \emph{Konrad-Zuse-Zentrum f\"{u}r
Informationstechnik} (ZIB) and a \emph{Miller Research Fellowship}.}

\begin{abstract}
  An arrangement of finitely many tropical hyperplanes in the tropical torus~$\T^{d-1}$
  leads to a notion of `type' data for points in $\T^{d-1}$, with the underlying unlabeled
  arrangement giving rise to `coarse type'.  It is shown that the decomposition of
  $\T^{d-1}$ induced by types gives rise to minimal cocellular resolutions of certain
  associated monomial ideals.  Via the Cayley trick from geometric combinatorics this also
  yields cellular resolutions supported on mixed subdivisions of dilated simplices,
  extending previously known constructions.  Moreover, the methods developed lead to an
  algebraic algorithm for computing the facial structure of arbitrary tropical complexes
  from point data.
\end{abstract}

\maketitle

\section{Introduction}

\noindent
The study of convexity over the tropical semiring has been an area of active research in
recent years.  Fundamental properties of tropical convexity, in particular from a 
combinatorial perspective, were established by Develin and Sturmfels in \cite{DevStu04}.
There the notion of a tropical polytope was defined as the tropical convex hull of a
finite set of points in the tropical torus $\T^{d-1}$.  Fixing the set of generating
points yields a decomposition of the tropical polytope called the \emph{tropical complex},
and in \cite{DevStu04} it was shown that the collection of such complexes are in bijection
with the regular subdivisions of a product of simplices.  The origin of tropical convexity
can be traced back to the study of `max-plus linear algebra'; see
\cite{LitvinovMaslovShpiz01, CohenGaubertQuadrat05} and the references therein for a
proper account.

A tropical complex can be realized as the subcomplex of bounded cells of the polyhedral
complex arising from an arrangement of tropical hyperplanes, and it is this
perspective that we adopt in this paper.  We study combinatorial properties of
arrangements of tropical hyperplanes, and in particular their relation to
algebraic properties of associated monomial ideals.  A tropical hyperplane in
$\T^{d-1}$ is defined as the locus of `tropical vanishing' of a linear form,
and can be regarded as a fan polar to a $(d-1)$-dimensional simplex.  In
this way each tropical hyperplane divides the ambient space $\T^{d-1}$ into
$d$ \emph{sectors}.  Given an arrangement $\A$ of tropical hyperplanes and a
point $p \in \T^{d-1}$, one can record the position of $p$ with respect to
each sector of each hyperplane.  This tropical analog of the covector data of
an oriented matroid is called the \emph{type} data, and the combinatorial
approach to tropical convexity taken in \cite{DevStu04} is based on this
concept.  Here we consider a coarsening of the type data (which we call
\emph{coarse type}) arising from an arrangement $\A$ of tropical hyperplanes,
amounting to neglecting the labels on the individual hyperplanes.

The connection between tropical polytopes/complexes and resolutions of monomial ideals was
first exploited by Block and Yu in \cite{BlockYu06}.  There the authors associate a
monomial ideal to a tropical polytope with generators in general position, and use
algebraic properties of its minimal resolution to determine the facial structure of the
bounded subcomplex.  Further progress along these lines was made in
\cite{DevelinYu07}. The primary tool employed in this context is that of a \emph{cellular
  resolution} of a monomial ideal $I$, where the $i$-th syzygies of $I$ are encoded by the
$i$-dimensional faces of a polyhedral complex (see Section \ref{sec:resolutions}). The
ideals from \cite{BlockYu06} are squarefree monomials ideals generated by the
\emph{cotype} data, i.e.~the complements of the tropical covectors, arising from the
associated arrangement of hyperplanes, and hence can be seen as a tropical analog of the
\emph{(oriented) matroid ideals} studied by Novik, Postnikov and Sturmfels in
\cite{NovPosStu}.

In this paper we study the polyhedral complex $\CD_\A$ (and its bounded subcomplex
$\boundedCD_\A$) induced by the type data of an arrangement $\A$ of $n$ hyperplanes in
$\T^{d-1}$.  Both complexes are naturally labeled by fine and coarse type and cotype data,
and we show how the resulting labeled complexes support minimal (co)cellular resolutions
of associated monomial ideals.  We pay special attention to labels given by coarse type.
For instance, we show that $\CD_\A$ supports a minimal \emph{cocellular} resolution of the
ideal $\cI$ generated by monomials corresponding to the set of all coarse types.  The
proof involves a consideration of the topology of certain subsets of $\CD_\A$ as well as
the combinatorial properties of the coarse type labelings.  When the arrangement $\A$ is
\emph{sufficiently generic} we show that the resulting ideal is always given by $\langle
x_1, \dots, x_d \rangle^n$, the $n$-th power of the maximal homogeneous ideal; in general,
$\cI$ is some Artinian subideal.  Our results in this area are all independent of the
characteristic of the coefficient field.

Via the connection to products of simplices and the Cayley trick we interpret these
results in the context of mixed subdivisions of dilated simplices.  In particular we
obtain a minimal \emph{cellular} resolution of $\cI$ supported on a subcomplex of the
dilated simplex $n \Delta_{d-1}$.  One other direct consequence is that \emph{any} regular
fine mixed subdivision of $n \Delta_{d-1}$ supports a minimal resolution of $\langle x_1,
\dots, x_d \rangle^n$.  This extends a result of Sinefakopoulos from \cite{Sine08} where a
particular subdivision is considered (although much less explicitly), and also complements
a construction of Engstr\"{o}m and the first author from \cite{DocEng08} where such
complexes are applied to resolutions of hypergraph edge ideals.  The duality between
tropical complexes and mixed subdivisions of dilated simplices was established in
\cite{DevStu04}, and we show how this extends to the algebraic level in terms of Alexander
duality of our resolutions of the coarse type and cotype ideals.

Finally, we show how these algebraic results lead to observations regarding the
combinatorics of tropical polytopes/complexes and mixed subdivisions of dilated simplices.
We obtain a formula for the $f$-vector of the bounded subcomplex of an arbitrary tropical
hyperplane arrangement in terms of the Betti numbers of the associated coarse type ideal.
The uniqueness of minimal resolutions also implies that for any sufficiently generic
arrangement $\A$, the multiset of coarse types is independent of the arrangement.
Furthermore, we present an algorithm for determining the incidence face structure of a
tropical complex from the coordinates of the generic set of vertices, utilizing the fact
that such a complex supports a minimal resolution of the square-free monomial cotype
ideal.  This approach was first introduced by Block and Yu in \cite{BlockYu06} for the
case of sufficiently generic arrangements, and we extend the algorithm to the general
case.

The rest of the paper is organized as follows.  In Section \ref{sec:coarse} we review the
basic notions of tropical convexity including tropical hyperplanes and type data, and
discuss the polytopal complex that arises from an arrangement of tropical hyperplanes. We
introduce the notion of \emph{coarse type} and establish some of the basic properties that
will be used later in the paper.  In Section \ref{sec:resolutions} we introduce the
monomial ideals that arise from an arrangement of hyperplanes and show that the polytopal
complexes labeled by fine and coarse (co)type support cocellular (and cellular)
resolutions.  In Section \ref{sec:Mixed} we recall the construction of the Cayley trick
and use this to interpret our results in terms of mixed subdivisions of dilated simplices.
In Section \ref{sec:examples} we discuss some examples of our construction, in particular
the case of the (generic) \emph{staircase triangulation} (recovering a result of
\cite{Sine08}) as well as a family of non-generic arrangements corresponding to tropical
\emph{hypersimplices}.  In Section \ref{sec:facecounting} we show our results give rise to
certain consequences for the combinatorics (e.g., the $f$-vector) of the bounded
subcomplexes of tropical hyperplane arrangements, and also describe an algorithm for
determining the entire face poset from the coordinates of the arrangement.  This
strengthens a result from \cite{BlockYu06}.  Finally, we end in Section~\ref{sec:final}
with some concluding remarks and open questions.

We are grateful to Kirsten Schmitz for very careful proof reading.

\section{Tropical convexity and coarse types}\label{sec:coarse}

\noindent
In order to fix our notation we begin in this section with a brief review of
the foundations of tropical convexity as layed out by Develin and
Sturmfels in~\cite{DevStu04}.  We then define `coarse types' and establish some
combinatorial and topological results regarding the type decomposition of the
tropical torus induced by a finite set of points.  While some of these
observations may be worthwhile in their own right, their main interest for us
will be their applications to subsequent constructions of (co)cellular
resolutions.

\subsection{Tropical convexity and tropical hyperplane arrangements}

Tropical convexity is concerned with linear algebra over the \emph{tropical
semi-ring} $(\R,\oplus,\odot)$, where
\[
    x\oplus y \ := \ \min(x,y) \quad \text{and} \quad x\odot y \ := \ x+y \, .
\]
We will sometimes replace the operation $\min$ with $\max$, and although the two resulting
semi-rings are isomorphic via \[ -\max(x,y) \ = \ \min(-x,-y) \, , \] it
will be useful for us to consider both structures on the set $\R$
simultaneously.  To avoid confusion we will therefore use the terms
\emph{min-tropical semi-ring} and \emph{max-tropical semi-ring}, respectively.
Componentwise tropical addition and tropical scalar multiplication turn $\R^d$
into a semi-module.  The \emph{tropical torus} $\T^{d-1}$ is the quotient of
Euclidean space $\R^d$ by the linear subspace $\R\vones$, where $\vones \in
\R^d$ is the all-ones vector. By interpreting this quotient in the category of
topological spaces, $\T^{d-1}$ inherits a natural topology which is
homeomorphic to the usual topology on $\R^{d-1}$.  A set $S\subset\T^{d-1}$ is
\emph{tropically convex} if it contains $(\lambda\odot x)\oplus(\mu\odot y)$
for all $x,y\in S$ and $\lambda,\mu\in\R$.  For an arbitrary set
$S\subset\T^{d-1}$ the \emph{tropical convex hull} $\tconv(S)$ is defined as
the smallest tropically convex set containing~$S$. If the set $S$ is finite,
then $\tconv(S)$ is called a \emph{tropical polytope}.  In this paper, tropical
convexity and related notions will be studied with respect to both $\min$ and
$\max$, and hence we will also talk about \emph{max-tropically convex} sets
and the like.

The \emph{tropical hyperplane} with \emph{apex} $-a\in\T^{d-1}$ is the
set
\[
H(-a) \ := \ \SetOf{p\in\T^{d-1}}{(a_1\odot p_1)\oplus(a_2\odot
  p_2)\oplus\dots\oplus(a_d\odot p_d) \text{ is attained at least twice}} \, .
\]
That is, a tropical hyperplane is the tropical vanishing locus of a polynomial homogeneous
of degree $1$ with real coefficients.  If we want to explicitly distinguish between the
$\min$- and $\max$-versions we write $H^{\min}(-a)$ and $H^{\max}(-a)$, respectively.  Any
two $\min$-tropical (resp.\ $\max$-tropical) hyperplanes are related by an ordinary
translation, and hence a tropical hyperplane is completely determined by its apex.
Furthermore, the $\min$-tropical hyperplane with apex $0$ is a mirror image of the
$\max$-tropical hyperplane with apex $0$ under the map $x\mapsto -x$. The complement of
any tropical hyperplane in $\T^{d-1}$ consists of precisely $d$ connected components, its
\emph{open sectors}.  Each open sector is convex, in both the tropical and ordinary sense.
The \emph{$k$-th (closed) sector} of the $\max$-tropical hyperplane with apex $a$ is the
set
\[
    S_k^{\max}(a) \ := \ \SetOf{p\in\T^{d-1}}{a_k-p_k\le a_i-p_i \text{ for
    all } i\in[d]} \, .
\]
Similarly we have
\[
    S_k^{\min}(a) \ := \ \SetOf{p\in\T^{d-1}}{a_k-p_k\ge a_i-p_i \text{ for
    all } i\in[d]}
\]
for the $\min$-version.  Notice that $x\in S_k^{\max}(y)$ if and only if $y\in
S_k^{\min}(x)$.  Each closed sector is the topological closure of an open
sector.  Again the closed sectors are tropically and ordinarily convex.  A
sequence of points $V=(v_1,v_2,\dots,v_n)$ in $\T^{d-1}$ gives rise to the
arrangement
\[
    \A(V) \ := \ \big(H^{\max}(v_1),H^{\max}(v_2),\dots,H^{\max}(v_n)\big)
\]
of $n$ labeled $\max$-tropical hyperplanes.  The position of points in
$\T^{d-1}$ relative to each hyperplane in the arrangement furnishes combinatorial data and leads to the following definition.

\begin{dfn}[Fine type]\label{dfn:cotype}
  Let $\A=\A(V)$ be the arrangement of $\max$-tropical hyperplanes given by
  $V=(v_1,v_2,\dots,v_n)$ in $\T^{d-1}$.  The \emph{fine type} (or sometimes simply
  \emph{type}) of a point $p \in \T^{d-1}$ with respect to $\A$ is the table $T_\A(p) \in
  \ZO^{n \times d}$ with \[ T_\A(p)_{ik} = 1 \quad \text{if and only if} \quad p \in
  S_k^{\max}(v_i) \] for $i\in[n]$ and $k\in[d]$.
  % Raman: There was only a single place were we used this notation.
  % Sometimes we will
  % identify a type $T$ with the sequence $(T_1,T_2,\dots,T_d)$, where
  % $T_k=\smallSetOf{i\in[n]}{T_{ik}=1}$.
  The \emph{fine cotype} $\overline{T}_\A(p) \in \ZO^{n \times d}$ is defined as the
  complementary matrix, that is,
  \[
  \overline{T}_\A(p)_{ik} = 1 \quad \text{if and only if} \quad T_\A(p)_{ik} = 0 \, .
  \]
\end{dfn}

Let us now fix a sequence of points $V$ and the corresponding arrangement $\A=\A(V)$ in
$\T^{d-1}$. We write $T(p)$ instead of $T_\A(p)$ when no confusion arises.  For a fixed
type $T = T(p)$, the set
\[
C_T^\circ \ := \ \SetOf{ q \in \T^{d-1} }{ T(q) = T }
\]
of points in $\T^{d-1}$ with that type is a relatively open subset of
$\T^{d-1}=\R^{d-1}$; called the \emph{relatively open cell} of type $T$.  The
set $C_T^\circ$ as well as its topological closure $C_T$ are tropically and
ordinarily convex.  The set of all relatively open cells
$\CD^\circ=\CD^\circ(V)$ partitions the tropical torus $\T^{d-1}$.  Moreover,
$C_T$, the \emph{closed cell of type $T$}, is the intersection of finitely
many closed sectors.  Since each closed sector is the intersection of finitely
many polyhedral cones, this implies that $C_T$, provided it is bounded, is a
polytope in both the classical and the tropical sense; in \cite{JoswigKulas08}
these were called \emph{polytropes}.  The collection of all closed cells
yields a polyhedral subdivision $\CD=\CD(V)$ of $\T^{d-1}$. The collection of
types with the componentwise order is anti-isomorphic to the face lattice of
$\CD(V)$, that is, $T_{\A}(D) \le T_{\A}(C)$ whenever $C \subseteq D$ are
closed cells of $\CD(V)$.  The cells which are bounded form the \emph{bounded
subcomplex} $\boundedCD=\boundedCD(V)$.  Following Ardila and Develin
\cite{AD09}, a fine type should be thought of as the tropical equivalent of a
covector in the setting of tropical oriented matroids, with the cells of
maximal dimension playing the role of the topes.  The following result highlights the
relation of $\min$-tropical polytopes and $\max$-tropical hyperplane
arrangements.

\begin{thm}[{\cite[Thm.~15 and Prop.~16]{DevStu04}}]\label{thm:tconv}
    The $\min$-tropical polytope $\tconv(V)$ is the union of cells in the
    bounded subcomplex $\boundedCD(V)$ of the cell decomposition of $\T^{d-1}$
    induced by the $\max$-tropical hyperplane arrangement $\A(V)$.
\end{thm}

The polytopal complex $\boundedCD(V)$ induced by types is a subdivision of the
tropical polytope $\tconv(P)$ called the \emph{tropical complex} generated
by $V$. The points $V$ are in tropically \emph{general position} if the
combinatorial type of $\CD(V)$ (or equivalently $\boundedCD(V)$) is invariant
under small perturbations.  In Section~\ref{sec:Mixed}, we will elaborate on the
important connection between tropical hyperplane arrangements and mixed
subdivisions of dilated simplices. In light of this connection, the points $V$
are in general position if the associated mixed subdivision is \emph{fine}.

At this point we introduce our running example, borrowed from
\cite[Ex.~10]{BlockYu06}.

\begin{example}\label{exmp:BY:cells}
    The points $v_1=(0,3,6)$, $v_2=(0,5,2)$, $v_3=(0,0,1)$, and $v_4=(1,5,0)$
    give rise to the max-tropical hyperplane arrangement shown in
    Figure~\ref{fig:BY:arrangement}.  It decomposes the tropical torus $\T^2$
    into $15$ two-dimensional cells, three of which are bounded.  Note
    that there is precisely one bounded cell of dimension one (incident with
    the $0$-cell $v_1$) which is maximal with respect to inclusion.  This
    shows that the polytopal complex $\boundedCD(V)$ need not be pure.
    Moreover, one can check that the dual regular subdivision of
    $\Delta_3\times\Delta_2$ is a triangulation, and hence these four points
    are sufficiently generic.
\end{example}

\subsection{Coarse types}\label{subsec:coarsetypes}

As we have seen, the fine type records the position of a point relative to a labeled tropical
hyperplane arrangement.  Neglecting the labels on the hyperplanes leads to the
following coarsening of the type information.

\begin{dfn}[Coarse type]
    Let $\A=\A(V)$ be an arrangement of $n$ $\max$-tropical hyperplanes in
    $\T^{d-1}$. The \emph{coarse type} of a point $p \in \T^{d-1}$ with
    respect to $\A$ is given by $\ct_\A(p)=(t_1,t_2,\dots,t_d) \in \N^d$ with
    \[
        t_k \ = \ \sum_{i=1}^n  T_\A(p)_{ik}
    \]
    for $k \in [d]$.
\end{dfn}

The coarse type entry $t_k$ records for how many hyperplanes in $\A$ the point $p$ lies in
the $k$-th closed sector.  Again we will write $\ct(p)$ when no confusion can arise.  By
definition $\ct$ is constant when restricted to a relatively open cell $C_T^\circ$.  The
induced map $\CD^\circ \rightarrow \N^d$ is also denoted by $\ct$.

\begin{figure}[htb]
    \includegraphics[scale=0.8]{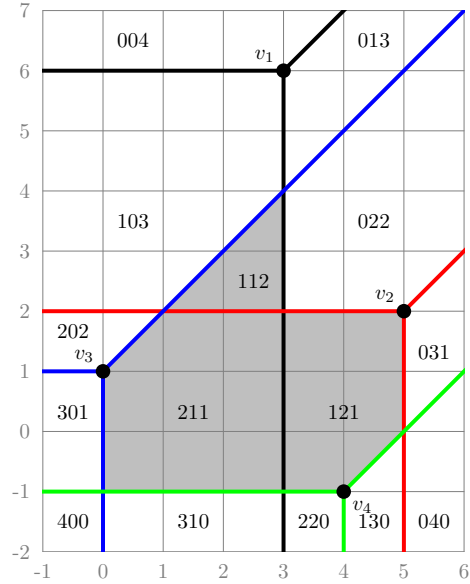}
    \caption{Coarsely labeled type decomposition of $\T^2$ induced by four
    max-tropical lines.  Bounded cells are shaded.}
    \label{fig:BY:arrangement}
\end{figure}

Objects in tropical geometry can be lifted to objects in classical algebraic
geometry by considering fields of formal Puiseux series and their valuations.
Following this path leads to another interpretation of the coarse types: For
$v_i \in \T^{d-1}$ the $\max$-tropical hyperplane $H^{\max}(v_i)$ is the
\emph{tropicalization} $\trop^{\max}(h_i)$ of a hyperplane given as the
vanishing of a homogeneous linear form $h_i\in\puiseux[x_1,x_2,\dots,x_d]$
defined over the field of Puiseux series. %(see \cite{XX}).
Let
\begin{equation}\label{eq:h}
  h \ = \ h_1 \cdot h_2 \cdots h_n
\end{equation}
be the product of these $n$ linear forms, so that $h$ is a homogeneous
polynomial of degree $n$ in $d$ indeterminates.  Notice that there is no
canonical choice for the polynomials $h_i$ and hence not for $h$.  In the
computation below we choose $h_i$ to be
$z^{-v_{i1}}x_1+z^{-v_{i2}}x_2+\dots+z^{-v_{id}}x_d$.

\begin{rem}
  Note that the classical Puiseux series have rational exponents, and
  therefore the valuation map takes rational values only.  For this reason the
  tropical hypersurface of a tropical polynomial is often defined as the
  topological closure of the vanishing locus of a tropical polynomial; e.g.,
  see Einsiedler, Kapranov, and Lind~\cite{EinsiedlerKapranovLind06}.  By
  extending $\puiseux$ to a field of \emph{generalized Puiseux series} one can
  directly deal with real exponents, as the corresponding valuation map is
  onto the reals; for a construction see Markwig~\cite{Markwig07}.  In the
  sequel we will make use of such a field of generalized Puiseux series, and
  we denote it by~$\generalizedpuiseux$.
\end{rem}

\begin{prop}
  The tropical hypersurface defined by $\trop^{\max}(h)$ is the union of the
  $\max$-tropical hyperplanes in $\A$.
\end{prop}

% Raman: Since we do not really explain trop(h) and this is just an
% alternative explanation of the coarse type, I think we do not include a
% proof of this standard fact.
%\begin{proof}
%  Each linear form $h_i$ defines a hyperplane in the projective space
%  $\PG_{d-1}\generalizedpuiseux$.  Their union is the ordinary hypersurface to the product
%  polynomial $h=h_1\cdot h_2 \cdots h_n$.  It follows that the tropical
%  hypersurface defined by $\trop^{\max}(h)$ is the union of the tropical hyperplanes
%  $\trop^{\max}(h_1),\trop^{\max}(h_2),\dots,\trop^{\max}(h_n)$.  The latter form the
%  arrangement~$\A$.
%\end{proof}

The tropical hypersurface defined by $\trop^{\max}(h)$ is the orthogonal
projection of the co\-di\-men\-sion-$2$-skeleton of an unbounded ordinary
convex polyhedron $P_\A$ in $\R^d$ whose facets correspond to monomials of the
polynomial $h$; see \cite[Thm.~3.3]{RichtergebertSturmfelsTheobald05}.

\begin{prop}\label{prop:coarse:exponent}
    Let $p\in\T^{d-1}\setminus\A$ be a generic point.  Then its coarse type
    $\ct_\A(p)$ is the exponent of the monomial in $h$ which defines the
    unique facet of $P_\A$ above $p$.
\end{prop}

\begin{proof}
    Since $p$ is generic, it is contained in a unique sector
    $S^{\max}_{\tau_i}(v_i)$ with respect to the hyperplane with apex $v_i$.
    Hence, $p$ satisfies the strict inequalities
    \begin{equation}\label{eq:type:ineq}
        v_{i,\tau_i}-p_{\tau_i} \ < \ v_{i,j}-p_j
    \end{equation}
    for all $i \in [n]$ and $j \in [d] \setminus \tau_i$. Equivalently,
    \[
        \sum_{i=1}^n p_{\tau_i} -v_{i,\tau_i}
    \]
    is the evaluation of the tropical polynomial $\trop^{\max}(h)$ and the
    corresponding monomial in $h$ at which the unique maximum is attained is
    $x_{\tau_1}x_{\tau_2}\cdots x_{\tau_n}$ with coefficient
    $z^{-v_{1,\tau_1}-v_{2,\tau_2}-\dots-v_{n,\tau_n}}$.
%
%   Let $p\in\T^{d-1}$ be a point not be contained in any $\max$-tropical
%   hyperplane from $\A$.  Equivalently, its fine type $T=T_\A(p)$ is a
%   partition $(T_1,T_2,\dots,T_d)$ of $[n]$, and the corresponding cell $X_T$
%   is maximal.  Hence the definition of the fine type gives the strict
%   inequalities
%   \begin{equation}\label{eq:type:ineq}
%       v_{i,k}-p_k \ < \ v_{i,j}-p_j
%   \end{equation}
%   for all $i\in T_k$ and all $j\in[d]$.   Let us write $\tau_i$ for the
%   unique $k\in[d]$ such that $i\in T_k$.  In view of the previous
%   inequalities the sum
%   \begin{equation}\label{eq:type:monomial}
%       \sum_{i\in[n]} v_{i,\tau_i}-p_{\tau_i}
%   \end{equation}
%   is minimal among all sums of $n$ arbitrary terms of type $v_{i,j}-p_j$.
%   Equivalently, its negative $\sum_{i\in[n]} (p_{\tau_i}-v_{i,\tau_i})$ is
%   maximal.  This expression is the evaluation of the tropical polynomial
%   $\trop^{\max}(h)$ at $p$.  The sum~\eqref{eq:type:monomial} corresponds to
%   the monomial $x_{\tau_1}x_{\tau_2}\dots x_{\tau_n}$, which has the
%   coefficient $z^{-v_{1,\tau_1}-v_{2,\tau_2}-\dots-v_{n,\tau_n}}$ in~$h$.
%   Now, since $p$ is generic the inequalities \eqref{eq:type:ineq} remain
%   valid for any point sufficiently close to $p$.  This shows that the facet
%   corresponding to $x_{\tau_1}x_{\tau_2}\dots x_{\tau_n}$ is the only facet
%   of $P_\A$ lying above $p$.
\end{proof}

\begin{example}[continued]\label{exmp:BY:poly}
    For the points in Example~\ref{exmp:BY:cells} we set
    \[
    \begin{array}{l@{\ \ = \ \ }r@{\ + \ }r@{\ + \ }r@{}l}
        h_1 & x_1       & z^{-3}x_2     & z^{-6}x_3&\,, \\
        h_2 & x_1       & z^{-5}x_2     & z^{-2}x_3&\,,  \\
        h_3 & x_1       & x_2           & z^{-1}x_3&\,, \ \text{and} \\
        h_4 & z^{-1}x_1 & z^{-5}x_2     & x_3&\,. \\
    \end{array}
    %\begin{array}{l@{\ \ = \ \ }l@{}l}
    %    h_1 & x_1+z^{-3}x_2+z^{-6}x_3&\,, \\
    %    h_2 & x_1+z^{-5}x_2+z^{-2}x_3&\,,  \\
    %    h_3 & x_1+x_2+z^{-1}x_3&\,, \ \text{and} \\
    %    h_4 & z^{-1}x_1+z^{-5}x_2+x_3&\,. \\
    %\end{array}
    \]
    Then $h=h_1\cdot h_2\cdot h_3\cdot h_4$ equals
    \begin{quote}
      $
        h = z^{-1}{x_1}^4+(z^{-6}+z^{-5}+z^{-4}+z^{-1}){x_1}^3{x_2}+(z^{-7}+z^{-3}+z^{-2}+1){x_1}^3{x_3}+(z^{-10}+z^{-9}+z^{-8}+z^{-6}+z^{-5}+z^{-4}){x_1}^2{x_2}^2+(z^{-12}+z^{-11}+3z^{-7}+2z^{-6}+2z^{-5}+2z^{-3}+1){x_1}^2{x_2}{x_3}+(z^{-9}+z^{-8}+z^{-6}+z^{-4}+z^{-2}+z^{-1}){x_1}^2{x_3}^2+(z^{-13}+z^{-10}+z^{-9}+z^{-8}){x_1}{x_2}^3+(z^{-16}+z^{-12}+2z^{-11}+2z^{-10}+z^{-9}+z^{-8}+z^{-7}+z^{-6}+z^{-5}+z^{-3}){x_1}{x_2}^2{x_3}+(2z^{-13}+z^{-12}+z^{-11}+z^{-9}+z^{-8}+z^{-7}+2z^{-6}+z^{-5}+z^{-4}+z^{-2}){x_1}{x_2}{x_3}^2+(z^{-10}+z^{-8}+z^{-7}+z^{-3}){x_1}{x_3}^3+z^{-13}{x_2}^4+(z^{-16}+z^{-14}+z^{-10}+z^{-8}){x_2}^3{x_3}+(z^{-17}+z^{-13}+2z^{-11}+z^{-9}+z^{-5}){x_2}^2{x_3}^2+(z^{-14}+z^{-12}+z^{-8}+z^{-6}){x_2}{x_3}^3+z^{-9}{x_3}^4
  \, .
  $
  \end{quote}
  Its $\max$-tropicalization is
  \begin{quote}
  $
  \trop^{\max}(h) = (-1\odot{x_1}^{\odot 4})\oplus(-1\odot{x_1}^{\odot 3}\odot{x_2})\oplus({x_1}^{\odot
    3}\odot{x_3})\oplus(-4\odot{x_1}^{\odot 2}\odot{x_2}^{\odot 2})\oplus({x_1}^{\odot
    2}\odot{x_2}\odot{x_3})\oplus(-1\odot{x_1}^{\odot 2}{x_3}^{\odot
    2})\oplus(-8\odot{x_1}\odot{x_2}^{\odot 3})\oplus(-3\odot{x_1}\odot{x_2}^{\odot
    2}\odot{x_3})\oplus(-2\odot{x_1}\odot{x_2}\odot{x_3}^{\odot
    2})\oplus(-3\odot{x_1}\odot{x_3}^{\odot 3})\oplus(-13\odot{x_2}^{\odot
    4})\oplus(-8\odot{x_2}^{\odot 3}\odot{x_3})\oplus(-5\odot{x_2}^{\odot
    2}\odot{x_3}^{\odot 2})\oplus(-6\odot{x_2}\odot{x_3}^{\odot
    3})\oplus(-9\odot{x_3}^{\odot 4})
  \, .
  $
  \end{quote}
  The polynomial $h$ and its $\max$-tropicalization have $15$ terms each, one
  for each maximal cell of the arrangement shown in
  Figure~\ref{fig:BY:arrangement}.  For instance, the point $p=(0,1,0)$ is
  generic with fine type $T(p)=(12,3,4)$.  Evaluating $\trop^{\max}(h)$ at $p$
  gives $-1$, and this maximum is attained for the unique tropical monomial
  ${x_1}^{\odot 2}\odot{x_2}\odot{x_3}=2x_1+x_2+x_3$.  We have
  $\ct(p)=(2,1,1)$ for the coarse type.
\end{example}

Recall that a \emph{composition} of $n$ into $d$ parts is a sequence
$(t_1,t_2,\dots,t_d)\in\N^d$ of nonnegative integers such that
$t_1+t_2+\dots+t_d = n$. Such compositions bijectively correspond to monomials
of total degree $n$ in $d$ variables, of which there are exactly
$\tbinom{n+d-1}{n}=\tbinom{n+d-1}{d-1}$.

\begin{thm}\label{thm:t_injective}
  For an arrangement $\A = \A(V)$, the map $\ct$ from the set of cells in $\CD_\A$ of
  maximal dimension $d-1$ to the set of compositions of $n$ into $d$ parts is injective.
  Moreover, if the points $V$ are sufficiently generic then the map $\ct$ is bijective.
\end{thm}
\begin{proof}
  The injectivity of $\ct$ follows immediately from Proposition~\ref{prop:coarse:exponent}.

  Now suppose that the points in $V$ are sufficiently generic.  All
  coordinates are finite, and hence all linear monomials are present in the
  linear form $h_i$ with non-zero coefficients.  Since $h$ is the product of
  $h_1,h_2,\dots,h_n$ all possible monomials of degree $n$ in the $d$
  indeterminates $x_1,x_2,\dots,x_d$ actually occur in $h$.  Therefore all
  tropical monomials occur in the tropicalization $\trop^{\max}(h)$.  As we
  assumed that the coefficients in $V$ are chosen sufficiently generic, the
  coefficients in the tropical polynomial $\trop^{\max}(h)$ are generic.
  Equivalently, each monomial of $h$ defines a facet of $P_\A$, and hence the
  claim.
\end{proof}

Proposition~\ref{prop:latticepoints} provides an interpretation of the above result in
terms of mixed subdivisions of dilated simplices.

\begin{cor}
  For an arrangement $\A = \A(V)$, the number of cells in $\CD_\A$ of maximal dimension
  $d-1$ does not exceed $\tbinom{n+d-1}{n}$. If the points $V$ are sufficiently generic
  then this bound is attained.
\end{cor}

Note that the injectivity of $\ct$ does not extend to the lower-dimensional cells, even in the sufficiently generic case.  For instance, in the arrangement considered in
Examples \ref{exmp:BY:cells} and~\ref{exmp:BY:poly} the points $(0,2,3)$ and $(0,3,5)$ lie
in the relative interiors of two distinct $1$-cells; yet they share the same coarse type
$(1,1,3)$.  However coarse types are \emph{locally} distinct in the following sense.

\begin{prop}\label{prop:minimal}
  Let $C$ and $D$ be two distinct cells in $\CD_{\A}$. If $C$ is contained in the closure
  of $D$ then $C$ and $D$ have distinct coarse types.
\end{prop}
\begin{proof}
    By \cite[Cor.~13]{DevStu04} we have $\ft_{\A}(D) \le \ft_{\A}(C)$ and
    $\ft_{\A}(D) \not= \ft_{\A}(C)$. This, in particular implies that $C$ is
    contained in more closed sectors with respect to $\A$ and hence $C$ and
    $D$ cannot have the same coarse type.
\end{proof}

Note also that the boundedness of a cell in $\CD_{\A}$ can be read off its
coarse type.

\begin{prop}[{\cite[Cor.~12]{DevStu04}}]\label{prop:bounded_coarse}
    Let $\A = \A(V)$ be a tropical hyperplane arrangement and let $C \in
    \CD_{\A}$ be a cell in the induced decomposition. Then $C$ is bounded if
    and only if $\ct(C)_i > 0$ for all $i = 1,\dots,n$.
\end{prop}
\begin{proof}
    We noted previously that for two points $p,q \in \T^{d-1}$ we have $p \in
    S_k^{\max}(q)$ if and only if $q \in S_k^{\min}(p)$. Now, the point $p$ is
    contained in an unbounded cell of $\CD_{\A}$ if there is a $k$ such that
    $S_k^{\min}(p) \cap V = \emptyset$. This is the case if and only if
    $t_\A(p)_k = 0$, that is, there is no hyperplane $H_i$ for which $p$ is
    contained in the $k$-th sector.
\end{proof}

\begin{rem}
    There is another perspective on Proposition~\ref{prop:coarse:exponent}
    that we wish to share. The standard valuation $\nu :
    (\generalizedpuiseux)^* \rightarrow \R$ on the field of (generalized)
    Puiseux series maps an element to its order.
    For $\omega \in \R^d$, the
    \emph{$\omega$-weight} of $f = \sum_a c_a \x^a \in
    \generalizedpuiseux[x_1,\dots,x_d]$ is $\omega(f) = \max\{ \nu(c_a) +
    \langle \omega, a \rangle : c_a~\not=~0\}$ and the \emph{initial form} of
    the polynomial $f$ is given by
    $\iform_\omega(f) := \sum \{ \x^a : \nu(c_a) + \langle \omega, a \rangle =
    \omega(f) \}$. Finally, for an ideal $I \subset
    \generalizedpuiseux[x_1,\dots,x_d]$ the \emph{initial ideal}
    $\iform_\omega(I)$ is generated by all $\iform_\omega(f)$ for $f \in I$.
    An alternative characterization of the tropical variety was given by
    Speyer and Sturmfels~\cite[Thm.~2.1]{spey04} as the collection of weights
    $\omega \in \R^d$ such that $\iform_\omega(I)$ does not
    contain a monomial. For a tropical hyperplane arrangement $\A = \A(V)$
    with defining polynomial $h = h_\A \in \generalizedpuiseux[x_1,\dots,x_d]$
    the coarse type $\ct(p)$ of a generic point $p\in\T^{d-1}\setminus\A$
    satisfies
    \[
        \ideal{\x^{\ct(p)}} \ = \ \iform_p(\ideal{h}) \, .
    \]
\end{rem}

\begin{rem}
  The evaluation of the tropical polynomial $\trop^{\max}(h)$ at some point $p$ can be
  modeled as a min-cost-flow problem as follows: Define a directed graph with nodes
  $\alpha$, $\beta_1,\beta_2,\dots,\beta_n$, $\gamma_1,\gamma_2,\dots,\gamma_d$, and
  $\delta$.  For each $i\in[n]$ and each $k\in[d]$ there are the following arcs: from
  $\alpha$ to $\beta_i$ with cost $0$ and capacity $1$, from $\beta_i$ to $\gamma_k$ with
  cost $v_{ik}$ and capacity $1$, from $\gamma_k$ to $\delta$ with cost $p_k$ and
  unbounded capacity.  Now $\trop^{\max}(h)(p)$ equals the minimal cost of shipping $n$
  units of flow from the unique source $\alpha$ to the unique sink~$\delta$.  This problem
  has a strongly polynomial time solution in the parameters $n$ and $d$; see
  \cite[Chapter~12]{Schrijver03:CO_A}.  This is remarkable in view of the fact that the
  tropical polynomial $\trop^{\max}(h)$ may have exponentially many monomials.
\end{rem}

\subsection{Topology of types}\label{subsec:top_coarsetypes}

We next investigate topological properties of certain subsets of $\T^{d-1}$
induced by a tropical hyperplane arrangement.  The motivation for such a study
comes from our applications to (co)cellular resolutions as described in
Section~\ref{sec:resolutions}.  However, since the methods used to establish
the desired properties are based on notions from (coarse) tropical convexity,
we include them here.  In particular, we show that certain subsets of
$\T^{d-1}$ are contractible by proving that they are, in fact, tropically
convex. Let us emphasize that none of the results that follow require the
hyperplanes to be in general position.

We begin with the following observation, which was established in
\cite[Thm.~2]{DevStu04}.  We include a short proof for the sake of
completeness.

\begin{prop}\label{prop:convexcont}
  A tropically convex set is contractible.
\end{prop}

\begin{proof}
  The \emph{distance} of two points $p,q\in\T^{d-1}$ is defined as
  \[
  \dist(p,q) \ := \ \max_{1\le i<j \le d} | p_i-p_j+q_j-q_i | \, .
  \]
  We have the equations
  \[
  \big(\dist(p,q) \odot p \big) \oplus q \ = \ q \quad \text{and} \quad p \oplus \big(\dist(p,q)\odot
  q \big) \ = \ p \, .
  \]
  Now let $p$ be a point in some tropically convex set $S$.  The map
  \[
  \eta \,:\, S \times[0,1] \to S \,:\, (q,t) \mapsto \big(((1-t)\cdot\dist(p,q) ) \odot p \big)
  \oplus \big((t\cdot\dist(p,q))\odot q \big)
  \]
  is continuous, and it contracts the set $S$ to the point $p$.
\end{proof}

Recall that for two types $T, T^\prime \in \ZO^{n \times d}$ we write $T \le
T^\prime$ for the componentwise induced partial order.  We let $\min(T,T^\prime)$ and
$\max(T,T^\prime)$ denote the tables with entries given by the componentwise
minimum and maximum, respectively.

\begin{prop}\label{prop:typeineq}
    Let $\A = \A(V)$ be a $\max$-tropical hyperplane arrangement in $\T^{d-1}$
    and $p,q \in \T^{d-1}$. Then
    \[
        \min( \ft_\A(p), \ft_\A(q) ) \ \le\  \ft_\A(r) \ \le \
        \max( \ft_\A(p), \ft_\A(q) )
    \]
    for every point $r \in \tconv^{\max}\{p,q\}$
    on the max-tropical line segment between $p$ and $q$.
%   Let $\A = \A(V)$ be a max-tropical hyperplane arrangement in $\T^{d-1}$.
%   For $p,q \in \T^{d-1}$ let $r \in \tconv^{\max}\{p,q\}$ be an arbitrary
%   point on the max-tropical line segment between $p$ and $q$. Then
%   \[
%       \min( \ft_\A(p), \ft_\A(q) ) \ \le\  \ft_\A(r) \ \le \
%       \max( \ft_\A(p), \ft_\A(q) ) \ .
%   \]
\end{prop}
\begin{proof}
  Let $r = (\lambda \odot p) \oplus (\mu \odot q) = \max\{ \lambda \vones + p, \mu \vones
  + q \}$ for $\lambda,\mu \in \R$ and let $k \in [d]$ be arbitrary but fixed.  We treat
  each inequality separately but in each case we assume without loss of generality that
  $r_k = \lambda + p_k \ge \mu + q_k$.

    For the first inequality suppose that both $p$ and $q$ are in the $k$-th sector of some hyperplane $H(v)$, so that $p_k - p_i \geq v_k - v_i$ and $q_k - q_i \geq v_k - v_i$ for all $i \in [d]$.  Now, for $j \in [d]$ we distinguish two cases.  If $r_j = \lambda + p_j \geq \mu + q_j$, then $r_j - r_k = p_j - p_k$, so that $r$ is in the $k$-th sector of $H(v)$.  If $r_j = \mu + q_j \geq \lambda + p_j$, then $r_k - r_j \geq \mu + q_k - r_j = \mu + q_k - (\mu + q_j) = q_k - q_j$.  Hence $r_k - r_j \geq v_k - v_j$ for all $j \in [d]$, and we conclude that $r$ is in the $k$-th sector of $H(v)$.

    For the second inequality suppose that $r$ is contained in the $k$-th sector of some hyperplane $H(v)$, so that $r_k - r_j \geq v_k - v_j$ for all $j \in [d]$.  Since $r_k = \lambda + p_k \geq \mu + q_k$ we have that $\lambda + p_k \geq v_k - v_j + r_j$ for all $j \in [d]$.   Also, $r_j \geq \lambda + p_j$ and hence $\lambda + p_k \geq v_k - v_j + \lambda + p_j$.  We conclude $p_k - p_j \geq v_k - v_j$ for all $j \in [d]$.  Hence $p$ is in the $k$-th sector of $H(v)$, as desired.
\end{proof}

From the definition the coarse type we obtain the following statement
regarding coarse types.

\begin{cor}\label{cor:coarse_segment_upper}
    Let $\A = \A(V)$ be a $\max$-tropical hyperplane arrangement in $\T^{d-1}$
    and $p,q \in \T^{d-1}$. Then
    \[
    \ct_\A(r) \ \le \ \max( \ct_\A(p), \ct_\A(q) )
    \]
    for  $r \in \tconv^{\max}\{p,q\}$.
\end{cor}

From Proposition~\ref{prop:typeineq} and Corollary~\ref{cor:coarse_segment_upper} we
obtain the following result regarding the topology of regions of bounded fine and coarse
(co)type.  These will be the main tools for establishing results regarding (co)cellular
resolutions in Section \ref{sec:resolutions}.

\begin{cor}\label{cor:downsets}
  Let $\A = \A(V)$ be an arrangement of $n$ $\max$-tropical hyperplanes in $\T^{d-1}$, and
  let $B \in \ZO^{n \times d}$ and $\b \in \N^d$.  With labels determined by fine
  (respectively, coarse) type the following subsets of $\T^{d-1}$ are $\max$-tropically
  convex and hence contractible:
  \[
  \begin{array}{lllll}
    (\CD_\A,\ft)_{\le B} &:=& \SetOf{ p \in \T^{d-1} }{ \ft_\A(p) \le B }
    &=& \bigcup\SetOf{ C \in \CD_\A }{ \ft_\A(C) \le B }\\
    (\CD_\A,\ct)_{\le \b} &:=& \SetOf{ p \in \T^{d-1} }{ \ct_\A(p) \le \b }
    &=& \bigcup\SetOf{ C \in \CD_\A }{ \ct_\A(C) \le \b }\,.\\
  \end{array}
  \]
  Similarly, with labels determined by fine (respectively, coarse) \emph{co}type the
  following subsets of $\T^{d-1}$ are $\min$-tropically convex and hence contractible:
  \[
  \begin{array}{lllll}
    (\CD_\A,\overline{\ft})_{\le B} &:=& \SetOf{ p \in \T^{d-1} }{\overline{\ft}_\A(p) \le B }
    &=& \bigcup\SetOf{ C \in \CD_\A }{\overline{\ft}_\A(C) \le B }\\
    (\CD_\A,\overline{\ct})_{\le \b} &:=& \SetOf{ p \in \T^{d-1} }{\overline{\ct}_\A(p) \le \b }
    &=& \bigcup\SetOf{ C \in \CD_\A }{\overline{\ct}_\A(C) \le \b }\,.\\
  \end{array}
  \]

  As a consequence, the two subsets of $\T^{d-1}$ obtained by replacing the complex
  $\CD_\A$ in the above pair of formulas with the bounded complex $\boundedCD_\A$ are
  $\min$-tropically convex and hence contractible.
\end{cor}

\begin{proof}
  The $\max$-tropical convexity of $(\CD_\A,\ft)_{\le B}$ follows from
  Proposition~\ref{prop:typeineq}, and Corollary~\ref{cor:coarse_segment_upper}
  establishes the same property for the coarse variant $(\CD_\A,\ct)_{\le \b}$.

  If $r$ is a point in the $\min$-tropical line segment between $p$ and $q$ a reasoning
  similar to the proof of Proposition~\ref{prop:typeineq} shows that $T_\A(r) \geq
  \min(T_\A(p), T_\A(q))$.  Passing to complements yields $\overline{T}_\A(r) \leq
  \max(\overline{T}_\A(p), \overline{T}_\A(q))$ and thus $\overline{\ct}_\A(r) \leq
  \max(\overline{\ct}_\A(p), \overline{\ct}_\A(q))$ for the coarse types.  We conclude
  that the sets $(\CD_\A,\overline{\ft})_{\le B}$ and $(\CD_\A,\overline{\ct})_{\le \b}$
  are $\min$-tropically convex.  For the last claim, we note that the tropical complex
  $\boundedCD_{\A}$ itself is $\min$-tropically convex; in view of
  Proposition~\ref{prop:bounded_coarse} it is a down-set of $\CD_\A$ under the cotype
  labeling.  Hence both $(\boundedCD_{\A},\overline{\ft})_{\leq B}$ and
  $(\boundedCD_{\A},\overline{\ct})_{\leq \b}$ are intersections of $\min$-tropically
  convex sets.
\end{proof}

Note that the sets of bounded (coarse) type need not be closed or bounded.  In
Figure~\ref{fig:BY:downset} we illustrate an example of a coarse down-set from our running
example.

\begin{figure}[htb]
    \includegraphics[scale=0.8]{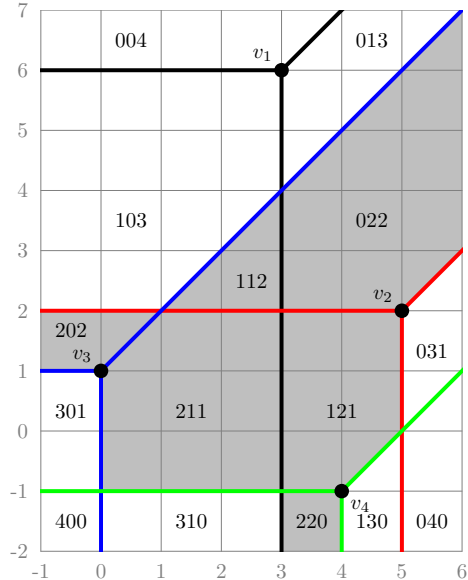}
    \caption{Coarse down-set $(C_\A)_{\le (2,2,2)}$ for the arrangement $\A$ from
      Example~\ref{exmp:BY:cells}. }
    \label{fig:BY:downset}
\end{figure}

%For the case of fine cotypes we have a similar result.

%\begin{cor}\label{cor:fine_acyclic}
%  Let $\A = \A(V)$ be a max-tropical hyperplane arrangement of $n$ hyperplanes in $\T^{d-1}$ and
%  consider the bounded complex $\boundedCD_{\A}$ labeled by fine
%  \emph{co}types.  Then for any $\b \in \N^{n \times d}$ the complex
%  $(\boundedCD_{\A})_{\leq \b}$ is min-tropically convex and hence
%  contractible.
%\end{cor}
%
%\begin{proof}
%    Let $p, q \in \T^{d-1}$ and let $r$ be a point in the min-tropical convex
%    hull of $p$ and $q$.  It follows from Proposition \ref{prop:typeineq} that
%    \[
%      \fct_\A(r)_{ik} \ \le\ \max\{ \fct_\A(p)_{ik}, \fct_\A(q)_{ik} \} \, .
%    \]
%    and hence for every $\b \in \N^{n \times d}$, the set $(\CD_\A)_{\le \b}$
%    is min-tropically convex.
%
%    The bounded complex $\boundedCD_{\A}$ itself is $\min$-tropically convex,
%    and so $(\boundedCD_{\A})_{\leq \b} = (\CD_\A)_{\leq \b} \cap
%    \boundedCD_{\A}$ is the intersection of two min-tropically convex sets,
%    and hence contractible.
%\end{proof}
%

\section{Resolutions}
\label{sec:resolutions}

\noindent
In this section we show how the polyhedral complexes $\CD_\A$ and
$\boundedCD_{\A}$ arising from a tropical hyperplane arrangement $\A = \A(V)$
support resolutions for associated monomial ideals.  We begin with a few
definitions.

\begin{dfn}\label{def:ideals}
    Let $\A = \A(V)$ be an arrangement of $n$ tropical hyperplanes in
    $\T^{d-1}$. The \emph{fine type} and \emph{fine cotype ideal} associated to
    $\A$ are the squarefree monomial ideals
    \begin{align*}
        \fI &\ =\ \ideal{ \x^{\ft(p)} : p \in \T^{d-1} } \ \subset \
        \field[x_{11},x_{12},\dots,x_{nd}]\\
        \fcI &\ =\ \ideal{ \x^{\fct(p)} : p \in \T^{d-1} } \ \subset \
        \field[x_{11},x_{12},\dots,x_{nd}]
    \end{align*}
    where $\x^{T(p)} = \prod\{ x_{ij} : T(p)_{ij} = 1\}$.
    Analogously, the \emph{coarse type} and \emph{coarse cotype ideal} associated to $\A$ are given by
    \begin{align*}
        \cI &\ =\ \ideal{ \x^{\ct(p)} : p \in \T^{d-1} } \ \subset \
        \field[x_1,x_2,\dots,x_d]\\
        \ccI &\ =\ \ideal{ \x^{\cct(p)} : p \in \T^{d-1} } \ \subset \
        \field[x_1,x_2,\dots,x_d].
    \end{align*}
    where $\x^{\ct(p)} = x_1^{t_1} x_2^{t_2} \cdots x_d^{t_d}$ with
    $\ct(p)=(t_1,t_2,\dots,t_d)$.
%
%   The \emph{coarse type ideal} associated to $\A$ is the
%   monomial ideal
%   \[
%       \cI \ =\ \ideal{ \x^{\ct(p)} : p \in \T^{d-1} } \ \subset \
%       \field[x_1,\dots,x_d]
% \]
%   where $\x^{\ct(p)} = x_1^{t_1} x_2^{t_2} \cdots x_d^{t_d}$ with
% $\ct(p)=(t_1,t_2,\dots,t_d)$.  The \emph{fine cotype ideal} associated to $\A$ is the
% squarefree monomial ideal
% \[
% \fcI \ =\ \ideal{ \x^{\overline{T}(p)} : p \in \T^{d-1} } \ \subset \
% \field[x_{11},x_{12},\dots,x_{nd}]
% \]
\end{dfn}

An analogous construction of fine cotype ideals for classical hyperplane arrangements
first appears in Novik et al.~\cite{NovPosStu} in the form of \emph{oriented matroid
  ideals}, where monomial generators are given by (complements of) covector data.  As
remarked earlier, the fine type of a tropical hyperplane arrangement can be thought of as
the covector data of a tropical oriented matroid and, from this perspective, a tropical
analog of \cite{NovPosStu} is given in \cite{BlockYu06}.  In Sections
\ref{sec:resolutions} and \ref{sec:facecounting} we will establish further connections
between our work and the results and constructions of \cite{BlockYu06}.  A common
generalization in the form of fan arrangements will be given in \cite{san09}.  In
\cite{NovPosStu} the authors also consider \emph{matroid ideals} given by specializing the
underlying oriented matroid ideal, in effect not distinguishing between different sides of
a (classical) hyperplane.  We point out in the case of our coarse type ideals, the
covector data coming from the different \emph{sectors} induced by a (tropical) hyperplane
(the tropical analog of the `side') is maintained, whereas the \emph{labels} on the
hyperplanes are no longer distinguished.

\subsection{Cellular and cocellular resolutions}\label{subsec:resolutions}

The relation between the decompositions of $\T^{d-1}$ and the
various ideals of Definition~\ref{def:ideals} is given via (co)cellular
resolutions.  Whereas cellular resolutions are by now a standard tool in
(combinatorial) commutative algebra, \emph{co}cellular resolutions seem to be
less popular.  As they arise naturally in connection with cotypes, we take the
opportunity to discuss cocellular resolutions alongside cellular resolutions
in some detail.   Our presentation is based on the book \cite{MilStu05} of Miller
and Sturmfels.

For a fixed field $\field$ we let $S = \field[x_1,\dots,x_m]$ be the
polynomial ring equipped with the $\Z^m$-grading given by $\deg \x^{\a} = \a
\in \Z^m$.  A free $\Z^m$-graded resolution $\res_\bullet$ of a $\Z^m$-graded
module $M$ is an algebraic complex of $\Z^m$-graded $S$-modules
\[
    \res_\bullet:\ \
    \cdots \stackrel{\phi_{k+1}}{\longrightarrow}
    F_k
    \stackrel{\phi_k}{\longrightarrow}
    \cdots
    \stackrel{\phi_2}{\longrightarrow}
    F_1
    \stackrel{\phi_1}{\longrightarrow}
    F_0
    \rightarrow 0
\]
where $F_i \cong \bigoplus_{\a \in \Z^m} S(-\a)^{\beta_{i,a}}$ are free $\Z^m$-graded
$S$-modules, the maps $\phi_i$ are homogeneous, and such that the complex is exact except
for $\operatorname{coker} \phi_1 \cong M$.  The resolution is called \emph{minimal}
exactly when $\beta_{i,a} = \dim_{\field} \operatorname{Tor}^S_i(S/I,\field)_a$ and the
numbers $\beta_{i,a}$ are called the \emph{fine graded Betti numbers}.

An efficient way of encoding $\Z^m$-graded resolutions of monomial ideals is
given by \emph{cellular} and \emph{cocellular resolutions}, introduced in
\cite{BayStu98} and \cite{Mil98}, respectively. Let $\P$ be an oriented
polyhedral complex and let $(\a_H)_{H \in \P} \in \Z^m$ be a labeling of the
cells of $\P$ such that
\[
    \a_H \ = \ \max\SetOf{\a_G}{\text{ for $G \subset H$ a face}} \, .
\]
The labeled complex $(\P,\a)$ gives rise to an algebraic complex of free
$\Z^m$-graded $S$-modules in the following way: Let
$(C_\bullet,\partial_\bullet)$ be the cellular chain complex for $\P$ and for
two cells $G, H \in \P$ with $\dim H = \dim G + 1$ denote by $\varepsilon(H,G)
\in \{0, \pm 1\}$ the coefficient of $G$ in the cellular boundary of $H$.  Now
define free modules
\[
    F_i \ := \ \bigoplus_{H \in \P,\, \dim H = i+1} S(-\a_H) \, .
\]
The differentials $\phi_i: F_i \rightarrow F_{i-1}$ are given on generators by
\[
    \phi_i(e_H) \ := \ \sum_{\dim G = \dim H -1 } \varepsilon(H,G) \x^{\a_H -
    \a_G} e_G \, .
\]
It can be verified that this defines an algebraic complex $\res^{\P}_\bullet$.
For $\b \in \Z^m$ denote by $\P_{\le \b}$ the subcomplex given by all cells $H
\in \P$ with $\a_H \le \b$, that is $(\a_H)_i \le b_i$ for all $i \in [m]$.

\begin{lem}[{\cite[Prop.~4.5]{MilStu05}}]\label{lem:cellular_criteria}
    Let $\res_\bullet^{\P}$ be the algebraic complex obtained from the labeled
    polyhedral complex $(\P,\a)$.  If for every $\b \in \Z^n$ the subcomplex
    $\P_{\le \b}$ is acyclic over $\field$, then $\res_\bullet^{\P}$ resolves
    the quotient of $S$ by the ideal $\ideal{ \x^{\a_v} : v \in P \text{
    vertex}}$.  Furthermore, the resolution is minimal if $\a_H \not= \a_G$
    for any two faces $G \subset H$ with $\dim H = \dim G + 1$.
\end{lem}

The complex $\res^{\P}_\bullet$ is called a \emph{cellular resolution} if it
meets the criterion above, and we say that the polyhedral complex $\P$
\emph{supports} the resolution.  If the labeling is such that
\[
    \a_H \ = \ \max\SetOf{ \a_G }{ \text{ for $G \supset H$ a face} }
\]
then $\P$ is said to be \emph{colabeled} and gives rise to an algebraic
complex utilizing the cellular \emph{cochain} complex of $\P$.  For this, let
$\res^\bullet_{\P}$ denote the algebraic complex with free $S$-modules $F^i :=
F_i$ as defined above and differentials $\phi^i : F^{i-1} \rightarrow F^i$
with
\[
    \phi^i(e_H) \ := \ \sum_{\dim G = \dim H + 1 } \delta(H,G) \x^{\a_H - \a_G} e_G
\]
where $\delta(H,G)$ records the corresponding coefficient in the coboundary
map for $\P$. If the algebraic complex is acyclic, the resulting resolution is
called \emph{cocellular}.  For $\b \in \Z^m$ the collection $\P_{\le \b}$ of
relatively open cells $H$ with $\a_H \le \b$ is \emph{not} a subcomplex.
However, as a topological space it is the union of the relatively open stars
of cells $G$ for which $\a_G \le \b$ is minimal and the cochain complex of the
nerve is isomorphic to the degree $\b$ component of $\res_\P^\bullet$. This
yields an analogous criterion regarding the exactness of $\res_\P^\bullet$.
The proof of Lemma~\ref{lem:cellular_criteria} given in \cite{MilStu05}
essentially proves the following criterion.

\begin{lem}\label{lem:cocellular_criteria}
    If $\P_{\le \b}$ is acyclic over $\field$ for every $\b \in \Z^n$ then
    $\res^\bullet_{\P}$ resolves $S/I$ where $I = \ideal{ \x^{\a_H} : H \in \P
    \text{ maximal cell } }$.  The resolution is minimal if $\a_H \not= \a_G$
    for any two faces $G \subset H$ with $\dim H = \dim G + 1$.
\end{lem}

\subsection{Resolutions from the arrangement}
\label{subsec:type_resolutions}

As usual let $\A = \A(V)$ be a $\max$-tropical hyperplane arrangement in $\T^{d-1}$
and let $\CD_\A$ be the induced polyhedral decomposition of $\T^{d-1}$.  As we have seen, every cell in $\CD_\A$ is naturally assigned a matrix and a vector determined by its fine and coarse type, respectively.  The next result states that these assignments are actually \emph{colabelings} in the sense of
Section~\ref{subsec:resolutions}.

\begin{prop}\label{prop:type_labeling}
    Let $\A = \A(V)$ be an arrangement of tropical hyperplanes in $\T^{d-1}$. For every cell
    $C \in \CD_\A$ of codimension $\ge 1$ we have
    \begin{align*}
        \ft_{\A}(C) &\ =\  \max\{ \ft_{\A}(D) : C \subset D\}, \text{ and }\\
        \ct_{\A}(C) &\ =\  \max\{ \ct_{\A}(D)\,: C \subset D\}\,.
    \end{align*}
    Thus, both the fine type and the coarse type yield a colabeling for
    the complex $\CD_{\A}$.
\end{prop}
\begin{proof}
    As we remarked before $C \subseteq D$ implies $\ft_{\A}(D) \le
    \ft_{\A}(C)$ and thus we have to show that $\ft_{\A}(C)$ is not strictly
    larger. For $k \in [d]$ consider the set $Q_k \subset \R^d$ given by all
    points $x \in \R^d$ such that $x_k \le x_i$ for all $i \not= k$.
    The set $Q_k$ is an ordinary polyhedron of dimension $d$ and it can be
    checked that for $p \in \T^{d-1}$ we have that $p \in S^{\max}_k(v_i)$
    implies $p + Q_k \subseteq S_k(v_i)$.
    Now, since $\mathrm{codim}(C) \ge 1$ we get that $C + Q_k$ meets the star
    of $C$ thereby showing that if $T_{\A}(C)_{ik} = 1$ for some $i$, then
    there is a cell $D \supset C$ with $T_{\A}(D)_{ik} = 1$.

    The same argument proves the purported equality for the coarse type.
\end{proof}
%Anton:I thought the case of fine type was already in Develin-Sturmfels?  And didn't we already have a proof for coarse type?

As an immediate consequence we obtain sets of generators for the fine and
the coarse type ideals.

\begin{cor}\label{cor:coarse_min_gen}
    For a tropical hyperplane arrangement $\A$ both the fine and coarse type
    ideals are generated by monomials corresponding to the respective types on the inclusion-maximal cells
    of $\CD_{\A}$.
\end{cor}

We now have all the ingredients to establish our first main result regarding the relation
between fine/coarse type ideals and the polyhedral decomposition $\CD_\A$.

\begin{thm}\label{thm:cocellular}
  Let $\A = \A(V)$ be a tropical hyperplane arrangement, and let $\CD_{\A}$ be the
  decomposition of the tropical torus $\T^{d-1}$ induced by $\A$.  Then with labels given
  by fine type (respectively, coarse type) the labeled complex $\CD_{\A}$ supports a
  minimal cocellular resolution of the fine type ideal $\fI$ (respectively, the coarse
  type ideal $\cI$).
\end{thm}
\begin{proof}
  By Proposition~\ref{prop:type_labeling} the polyhedral complex $\CD_\A$ is colabeled by
  both fine and coarse type. It follows from Lemma~\ref{lem:cocellular_criteria}
  and Corollary~\ref{cor:downsets} that this yields a cellular resolution of the
  respective type ideal. The minimality is a consequence of
  Proposition~\ref{prop:minimal}.
\end{proof}

A key point is that all of the above is valid for point configurations $V$
which are not necessarily in general position.  In the case of hyperplanes in
general position, the coarse type ideal is well known and, in particular, is
independent of the choice of hyperplanes.

\begin{cor}\label{cor:cocellular_generic}
    Let $\A = \A(V)$ be a \emph{sufficiently generic} arrangement of $n$
    tropical hyperplanes in $\T^{d-1}$.  Then $\CD_{\A}$ supports a minimal
    cocellular resolution of
    \[
        \cI \ =\ \ideal{x_1,\dots,x_d}^n \,
    \]
    the $n$-th power of the homogeneous maximal ideal.
\end{cor}

Since the minimal resolution of an ideal is unique up to isomorphism, we have
that the resolution arising from Corollary~\ref{cor:cocellular_generic} is
always isomorphic (as a chain complex), to the well-known Eliahou-Kervaire
resolution \cite[\S2.3]{MilStu05}.  However,
Corollary~\ref{cor:cocellular_generic} shows that there is a multitude of
colabeled complexes coming from tropical hyperplane arrangements that give
rise to a cellular description of the minimal resolution of
$\ideal{x_1,\dots,x_n}^d$.

We next turn to the other class of ideals introduced in the beginning of this section, namely the fine and coarse \emph{co}type ideals.  As was the case with the type ideals, we first consider the labeling of the relevant complexes.

\begin{prop}\label{prop:cotype_labeling}
    Let $\A = \A(V)$ be an arrangement of tropical hyperplanes. For every cell
    $D \in \CD_\A$ of dimension $\ge 1$ we have
    \begin{align*}
        \fct_{\A}(D) &\ =\  \max\{ \fct_{\A}(C) : C \subset D\}, \text{ and }\\
        \cct_{\A}(D) &\ =\  \max\{ \cct_{\A}(C)\,: C \subset D\}\,.
    \end{align*}
    Thus, both the fine and coarse types yield labelings for
    the complex $\CD_{\A}$.
\end{prop}
\begin{proof}
    The assertion follows by showing that
    \[
        \ft_{\A}(D)  \ \ge \  \min\{ \ft_{\A}(C) : C \subset D\}\,.
    \]
    We mimic the argument in the proof of
    Proposition~\ref{prop:type_labeling}.  It follows from the
    full-dimensionality of $Q_k$ that $C + Q_k$ intersects $D$ for every cell
    $C \subset D$ and thus $\ft(D)_{ik} \ge \ft(C)_{ik}$ for all $C$ in the
    boundary of $D$.
\end{proof}

The previous proposition implies that the cotype ideals are generated by the
inclusion-\emph{minimal} faces (namely, the 0-dimensional cells) of $\CD_{\A}$.  As a
consequence (together with the last part of Corollary~\ref{cor:downsets}) we see that
resolutions of these ideals are supported on the collection of bounded faces of
$\CD_{\A}$, and we obtain our next main result.  This generalizes
\cite[Theorem~1]{BlockYu06}.

\begin{thm}
  Let $\A = \A(V)$ be an arrangement of tropical hyperplanes and let $\boundedCD_\A$ be
  the subcomplex of bounded cells of $\CD_{\A}$. Then $\boundedCD_\A$, with labels given
  by fine cotype (respectively, coarse cotype) supports a minimal resolution of the fine
  cotype ideal $\fcI$ (respectively, coarse cotype ideal $\ccI$).
\end{thm}

\begin{rem}
  For fixed $n$ and $d$ we have a ring map $\psi : \field[x_{11},\dots,x_{nd}] \rightarrow
  \field[x_1,\dots,x_d]$ determined by $x_{ij} \mapsto x_j$. This map takes fine (co)type
  ideals to coarse (co)type ideals.  In either case, both the fine and the coarse ideals
  are resolved by the same polyhedral complex, and one might hope that the resolution of
  the fine ideal descends to the coarse ideal via $\psi$. Such an approach is employed in
  \cite{NovPosStu} in the passage from the oriented matroid to the matroid ideal.  There
  it is shown that any oriented matroid ideal $J$ is Cohen-Macaulay, with $(x_{11}-x_{21},
  x_{21}-x_{22}, \dots, x_{n1}-x_{n2})$ forming a linear system of parameters (and hence a
  regular sequence) in $S/J$.  Unfortunately, this approach does \emph{not} work in our
  context. It turns out that the fine cotype ideals $\fcI$ are not in general
  Cohen-Macaulay.  Indeed, regarding $\fcI$ as a Stanley-Reisner ideal we see that
  $S/\fcI$ has dimension $dn-2$, whereas an application of the Auslander-Buchsbaum formula
  (and knowledge of the minimal resolution) shows that its depth is $dn - d$.  The
  coarsening sequence $\{x_{11} - x_{21}, \dots, x_{11} - x_{n1}, \dots, x_{1d}-x_{nd}\}$
  is also of length $dn-d$, but one can show that it is not a regular sequence (not even a
  system of parameters).
\end{rem}

\section{The mixed subdivision picture}\label{sec:Mixed}

\noindent
As mentioned in Section \ref{subsec:coarsetypes}, an arrangement $\A = \A(V)$ of $n$
hyperplanes in $\T^{d-1}$ gives rise to a regular subdivision of the product of two
simplices $\Delta_{n-1} \times \Delta_{d-1}$.  To see this, we let $\Delta_{n-1} \subset
\R^n$ denote the convex hull of the standard basis vectors $\{e_i: 1 \leq i \leq n\}$.  If
$V=(v_1, \dots, v_n)$ is the ordered sequence of the apices of the arrangement $\A$, we
lift each vertex $(e_i,e_j)$ of $\Delta_{n-1} \times \Delta_{d-1}$ to the height
$(v_i)_j$, the $j$-th coordinate of the $i$-th apex.  Taking the lower convex hull of the
resulting configuration of points induces a (by definition regular) subdivision of
$\Delta_{n-1} \times \Delta_{d-1}$.  A main result from \cite{DevStu04} is that the
combinatorial types of the tropical complexes generated by a set of $n$ vertices in
$\T^{d-1}$, seen as polytopal complexes, are in bijection with the regular polyhedral
subdivisions of $\Delta_{n-1} \times \Delta_{d-1}$.

At this point, the reader may find it useful to contemplate the schematic
diagram in Figure~\ref{fig:schema} below.  It sketches how various relevant
combinatorial concepts are related.  Again $V$ denotes our ordered sequence of
$n$ points in $\T^{d-1}$, which can be thought of as the apices of a max
hyperplane arrangement $\A$ (where the hyperplane with apex $v_i$ corresponds
to the tropical vanishing of the form $-v_i \odot x$).  We have
$\boundedCD_{\A}$ given by the $\min$-tropical convex hull $\tconv^{\min}(V)$,
which is the bounded part of the polyhedral complex $\CD_\A$ determined by the
coarse types of the arrangement $\A$.  The correspondence described above
(here denoted `DS') associates to $\boundedCD_{\A}$ a certain regular
polyhedral subdivision $\Delta_{\A}$ of $\Delta_{n-1} \times \Delta_{d-1}$;
this subdivision is dual to the complex $\CD_\A$.  The Cayley trick (described
in more detail below, here denoted `Cay') then associates to that regular
subdivision a regular \emph{mixed} subdivision $\Sigma_\A$ of the dilated
simplex $n \Delta_{d-1}$.  Moreover, we have the `switch' isomorphism obtained
by exchanging coordinates of the product; applying the Cayley trick then gives
a regular mixed subdivision $\Sigma^*_\A$ of $d \Delta_{n-1}$.  Finally,
observe that $n\Delta_{d-1}$ is the Newton polytope of the polynomial $h$
defined in \eqref{eq:h}.  The mixed subdivision $\Sigma_\A$ is the
\emph{privileged} regular subdivision of $n\Delta_{d-1}$ defined by the
coefficients of $\trop^{\max}(h)$.  The corresponding arrow in the diagram is
marked `New' for `Newton'.

\begin{figure}[hb]
  $\xymatrix{
    & V \ar[d] \ar[ddr]^{\textrm{New}} &\\
    & \boundedCD_{\A} = \tconv^{\min}(V) \ar[dl]_{{\rm bounded}} \ar[d]^{\textrm{DS}} &\\
    \CD = \CD^{\max}(V) \ar[r]^-{\textrm{dual}} & \Delta_{\A} \textrm{ of } \Delta_{n-1} \times
    \Delta_{d-1} \ar[r]^-{\textrm{Cay}} \ar[d]_{\cong}^{\textrm{switch}} & \Sigma_{\A} \textrm{ of } n \Delta_{d-1} \\
    & \Delta_{\A} \textrm{ of } \Delta_{d-1} \times \Delta_{n-1} \ar[r]_-{{\rm Cay}} &
    \Sigma^*_{\A} \textrm{ of } d \Delta_{n-1} }$
  \caption{Diagram explaining how the various combinatorial concepts are related.}
  \label{fig:schema}
\end{figure}
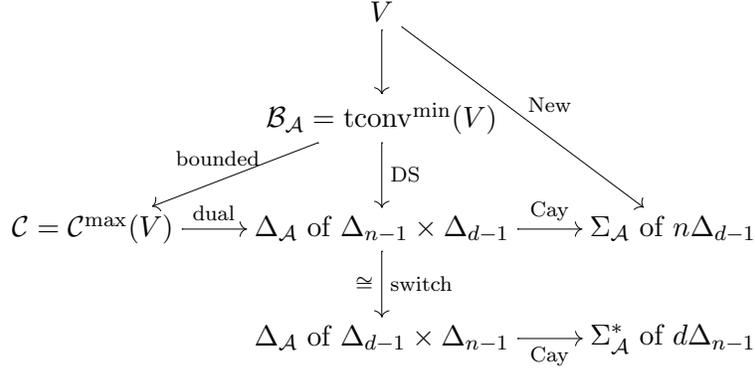

\subsection{The Cayley trick and mixed subdivisions} \label{sec:Cayley}

The \emph{Cayley trick} gives a bijection between the polyhedral subdivisions of
$\Delta_{n-1} \times \Delta_{d-1}$ and the \emph{mixed} subdivisions of $n \Delta_{d-1}$.
More generally, the Cayley trick relates mixed subdivisions of the Minkowski sum of
several polytopes $P_1,P_2, \dots, P_k$ with the polyhedral subdivisions of a certain
polytope known as the \emph{Cayley embedding} of those polytopes.  We refer to
\cite{Santos05} for a good account of the Cayley trick in the context of products of
simplices, but we wish to discuss some of the basics here.  This way it will also become
apparent in which way products of simplices arise.  The collection of all polyhedral
subdivisions of a polytope $P$ form a poset under refinement, and the minimal elements are
the triangulations of $P$.

Recall that if $P_1,P_2, \dots, P_n$ are polytopes in $\R^d$, then the \emph{Minkowski sum} is
defined to be the polytope
\[
P_1 + P_2 + \cdots + P_n \ := \ \big \{x_1 + x_2 + \cdots + x_n: x_i \in P_i \big \} \, .
\]
If $P = P_1 + P_2 + \cdots + P_n$ is a Minkowski sum of polytopes, a \emph{mixed cell} $B
\subseteq P$ is defined to be a Minkowski sum $B_1 + \cdots + B_n$, where each $B_i$ is a
polytope with vertices among those of $P_i$.  A \emph{mixed subdivision} of $P$ is then a
subdivision of $P$ consisting of mixed cells.  Once again, the mixed subdivisions of $P$
form a poset under refinement, and in this case the minimal elements are called
\emph{fine} mixed subdivisions.
% Michael: usually it improves readabilty if one writes also the second term of a sequence
% (not only the first and the last)

Now, if as above $P_1, P_2, \dots, P_n$ are polytopes in $\R^d$, we again let $\{e_1, \dots,
e_n\}$ be the standard basis of $\R^n$ and let $\iota_i: \R^d \rightarrow \R^d
\times \R^n$ denote the inclusion $x \mapsto (x,e_i)$.  The \emph{Cayley embedding} of
$P_1,P_2, \dots, P_n$ is defined to be
\[
\Cayley(P_1,P_2, \dots, P_n) \ := \ \conv \big(\bigcup \iota_i(P_i)\big) \, ,
\]
the convex hull of all these inclusions.  The intersection of $\Cayley(P_1,P_2, \dots,
P_n)$ with the affine subspace $\R^d \times \{\sum \frac{1}{n} e_i \}$ equals the (scaled)
Minkowski sum $(\frac{1}{n} \sum P_i)\times \{\sum \frac{1}{n} e_i \}$.  This way each
polyhedral subdivision of the Cayley embedding $\Cayley(P_1,P_2, \dots, P_n)$ induces a
polyhedral subdivision of the Minkowski sum $\sum P_i$.  The Cayley trick says that this
correspondence gives a poset isomorphism between the polyhedral subdivisions of
$\Cayley(P_1, \dots, P_n)$ and the mixed subdivisions of $\sum P_i$, with regular
subdivisions corresponding to coherent mixed subdivisions.

Here we have $P_1 = P_2 = \cdots = P_n = \Delta_{d-1}$, so that $\Cayley(\Delta_{d-1},
\dots, \Delta_{d-1}) = \Delta_{n-1} \times \Delta_{d-1}$, and hence the Cayley trick
provides a bijection between the subdivisions of the product $\Delta_{n-1} \times
\Delta_{d-1}$ and mixed subdivisions of the dilated simplex $n \Delta_{d-1}$.  In this
case, a mixed cell $\tau$ in a mixed subdivision of $n \Delta_{d-1}$ is given by $\tau =
\sum_{j=1}^n \Delta_{I_j}$ where each $\Delta_{I_j} = \conv \smallSetOf{ e_i }{ i \in I_j
}$ is a face of $\Delta_{d=1}$ with $I_j \subseteq [d]$.  Hence we can represent the cell
as simply $\tau = I_1 + I_2 + \cdots + I_n$.

We have seen that (regular) subdivisions of $\Delta_{n-1} \times \Delta_{d-1}$ are closely
related to the combinatorics of tropical complexes, and the aim of this section is to
interpret our results from the previous sections in terms of mixed subdivisions of dilated
simplices.  For this it will be convenient to have a notion of coarse type of a cell
defined purely in terms of the mixed subdivision.  Of course these are obtained by
applying the Cayley trick to the coarse types introduced above.

\begin{dfn}\label{dfn:coarsetypemixed}
  Suppose $\tau = I_1 + I_2 + \cdots + I_n$ is an $i$-dimensional mixed cell in a mixed
  subdivision of $n \Delta_{d-1}$, with each $I_j \subseteq [d]$.  The \emph{coarse type}
  $\ct(\tau) \in \N^d$ is a vector whose $i$-th coordinate is given by $\#
  \smallSetOf{I_j}{i \in I_j}$, the number of occurrences of $i$ in the decomposition of
  $\tau$.  The \emph{dual coarse type} $\dt(\tau) \in \N^n$ is a vector whose $i$-th
  coordinate is given by $\# I_i$, the number of elements of $I_i$.
\end{dfn}

The following can be seen as an interpretation of Theorem \ref{thm:t_injective} in the
context of mixed subdivisions.  Here we provide a complete proof here which does appeal to
tropical geometry and hence applies to the more general situation of not necessarily regular subdivisions.

\begin{prop}\label{prop:latticepoints}
  In any fine mixed subdivision $\Sigma$ of $n \Delta_{d-1}$, the set of $0$-dimensional
  cells are precisely the lattice points $n \Delta_{d-1} \cap \Z^{d}$, and the collection
  of coarse types of these cells are in bijection with the set of compositions of $n$ into
  $d$ parts.
\end{prop}

\begin{proof}
  As above, we denote by $\Delta_{d-1}=\conv\smallSetOf{e_i}{i\in[d]} \subset \R^d$ the
  standard simplex.  By definition, each cell $\tau$ of the subdivision $\Sigma$ is of the
  form $\tau = \sum_{j=1}^n \Delta_{I_j}$ where $I_j \subseteq [d]$ and $\Delta_{I_j} =
  \conv \smallSetOf{ e_i }{ i \in I_j }$ is a face of $\Delta_{d-1}$.  Since we assumed
  the mixed subdivision $\Sigma$ to be fine this yields $\dim \tau = \sum_j \dim
  \Delta_{I_j} = \sum_j (|I_j|-1)$.

  Any 0-dimensional cell of $\Sigma$ is of the form $\{e_{i_1}\} + \cdots + \{e_{i_n}\}$,
  where $1 \leq i_j \leq d$.  Any such cell is a lattice point in $n\Delta_{d-1} \cap
  \Z^{d}$, and the collection of coarse types of such cells correspond to the set of
  compositions of $n$ into $d$ parts.  In order to show that all such lattice points arise
  as 0-dimensional cells, let $t \in n\Delta_{d-1} \cap \Z^{d}$ and let $\tau =
  \sum_{j=1}^n \Delta_{I_j}$ be the inclusion minimal mixed cell containing $t$. By
  \cite[Prop.~14.12]{post09} every lattice point of $\tau$ is of the form $\sum_{j=1}^n
  \{e_{i_j}\}$ with $i_j \in I_j$ for all $j$. However, as the mixed subdivision is fine,
  $\tau$ is combinatorially isomorphic to the product $\prod_{j=1}^n \Delta_{I_j}$ and
  hence every sum of vertices is a vertex. Therefore, $\tau$ is a vertex of the mixed
  subdivision $\Sigma$.
\end{proof}

\begin{figure}[htb]
  \includegraphics[width=.4\textwidth]{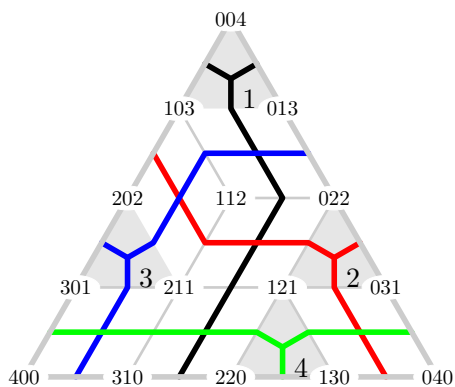}

  \caption{Mixed subdivision of $4\Delta_2$ corresponding to Example~\ref{exmp:BY:cells}.}
\end{figure}

\subsection{Resolutions supported by mixed subdivisions}

In this section we discuss our results regarding cellular resolutions in the context of mixed subdivisions of dilated simplices.  Although these results are more or less translations of the above via the
Cayley trick, we find it useful to make this transition explicit.  It seems that the
mixed subdivision picture allows for more natural statements whereas the tropical
convexity picture allows for more natural proofs.  We refer the reader back to
Definition~\ref{dfn:coarsetypemixed} for the definition of coarse type of a mixed cell.

\begin{cor}\label{cor:cellres}
  Let $\Sigma$ be any regular mixed subdivision of $n \Delta_{d-1}$.  Consider $\Sigma$ to
  be a labeled polytopal complex with each face $\sigma$ labeled by the least common
  multiple of the vertices that it contains.  Then for any field $\field$, the complex
  $\Sigma_{\A}$ supports a minimal cellular resolution of the coarse type ideal
  $\cI = \langle \x^{\ct(p)}: p \in \T^{d-1} \rangle$ in $\field[x_1,\dots,x_d]$.
\end{cor}

\begin{proof}
  The fact that $\Sigma_{\A}$ is a labeled complex follows from
  Proposition~\ref{prop:type_labeling}.  As a poset $\Sigma_{\A}$ is isomorphic to the
  corresponding regular subdivision of $\Delta_{n-1}\times\Delta_{d-1}$ via the Cayley
  trick.  By \cite[Lemma~22]{DevStu04} this regular subdivision is dual to the cell
  decomposition $\CD_\A$ of $\T^{d-1}$.  In this way the labeling of $\Sigma_{\A}$ turns
  into the colabeling of $\boundedCD_\A$ by coarse types.  Theorem~\ref{thm:cocellular}
  now establishes the claim.
\end{proof}

It is now straightforward to derive the mixed subdivision result corresponding to
Corollary~\ref{cor:cocellular_generic}.

\begin{cor}\label{cor:cellres_generic}
  Let $\Sigma$ be any \emph{fine} mixed subdivision of $n \Delta_{d-1}$.  Then
  $\Sigma_{\A}$, as a labeled polyhedral complex, supports a minimal cellular resolution
  of $\langle x_1, \dots, x_d \rangle^n$.
\end{cor}

\begin{proof}
  This follows from Corollary \ref{cor:cellres} and Proposition \ref{prop:latticepoints}.
\end{proof}

We note that mixed subdivisions of dilated simplices have been used to obtain cellular
resolutions in previous work.  In \cite{DocEng08} the authors study applications to
resolutions of edge ideals of graphs and hypergraphs.  In \cite{Sine08}, Sinefakopoulos
shows that the mixed subdivision of $n \Delta_{d-1}$ corresponding to the \emph{staircase}
triangulation of $\Delta_{n-1} \times \Delta_{d-1}$ supports a cellular resolution of
$\langle x_1, \dots, x_d \rangle^n$, although his construction is much less explicit.  We
discuss this example further in Section \ref{sec:examples}.

\begin{rem}
    It seems to be a challenging task to characterize which monomial ideals
    \mbox{$I
    \subset \field[x_1, \dots, x_d]$} arise as $\cI$ for some arrangement $\A$
    of tropical hyperplanes in $\T^{d-1}$.  Some necessary conditions are
    obvious, e.g., $I$ should be homogeneous of some degree $n$, and that
    $x_i^d$ should be contained in $I$ for all $i$.  In particular, this means
    that the coarse type ideal is necessarily Artinian.
\end{rem}

\subsection{Alexander duality of ideals and resolutions}

We have seen how the bounded subcomplexes of tropical hyperplane arrangements are related
to mixed subdivisions of dilated simplices in terms of a geometric duality.  This duality
extends to the algebraic level of our resolutions in the context of \emph{Alexander
  duality of resolutions}.  For this we will need the following notion of the Alexander
dual of a (not necessarily square-free) monomial ideal.

\begin{dfn}
  Suppose $I$ is a monomial ideal in the polynomial ring $\field[x_1, \dots, x_d]$ and let
  $\a \in \N^d$.  The \emph{Alexander dual of $I$ with respect to $\a$} is given by the
  intersection
  \[
  I^{[\a]} \ = \ \bigcap \SetOf{\m^{\a \backslash \b}}{{\bf x^b} \textrm{ is a minimal
      generator of $I$}} \, ,
  \]
  where $\a\backslash\b$ denotes the vector whose $i$-th coordinate is $a_i + 1 - b_i$ if
  $b_i \geq 1$, and is $0$ if $b_i = 0$.  Here we borrow the notation $\m^\a := \langle
  x_i^{a_i}:a_i \geq 1 \rangle$.
\end{dfn}

Note that if $I$ is a square-free monomial ideal (and hence the Stanley-Reisner ring of
some simplicial complex) and $\a = \vones$ is taken to be the all-ones vector, then this
notion recovers the more familiar notion of Alexander duality of simplicial complexes.
The main result concerning duality of resolutions, relevant for us, is the following
({\cite[Theorem~5.37]{MilStu05}}).

\begin{thm} \label{resduality} Suppose $I$ is a monomial ideal in degrees preceding some
  $\a \in {\N}^d$ and suppose $\res_\bullet^{\P}$ is a minimal cellular resolution of $S/(I +
  \m^{\a + \vones})$ such that all face labels on $\P$ precede $\a + \vones$.  Let $\Q$
  denote the labeled complex with the same underlying complex $\P$ but with labels
  $\overline \ct_F = \a + \vones - \ct_F$.  Then $\res_\bullet^{\Q_{\leq \a}}$ is a minimal cocellular resolution of $I^{[\a]}$.
  \end{thm}

Applying this theorem we obtain the following dual resolution of the coarse cotype ideal.
For $\sigma$ a face of a mixed subdivision $\Sigma$, the \emph{coarse cotype} of $\sigma$
is defined to be $n \vones - \ct(\sigma)$, where $\ct(\sigma)$ is the coarse type of
$\sigma$ defined in Definition~\ref{dfn:coarsetypemixed}.

\begin{prop} \label{dualres} Given any arrangement $\A$ of $n$ tropical hyperplanes in
  $\T^{d-1}$, let $\overline \Sigma_{\A}$ denote the associated mixed subdivision of $n
  \Delta_{d-1}$ with labels given by coarse cotype.  Then $\overline \Sigma_{\A}$ supports
  a minimal cocellular resolution of the coarse cotype ideal $\ccI$ in $\field[x_1, \dots,
  x_d]$. Consequently the associated tropical complex $\boundedCD_{\A}$, with
  labels given by coarse cotype, supports a minimal \emph{cellular} resolution of $\ccI$.
\end{prop}

\begin{proof}
We apply Theorem \ref{resduality} with $\a := (n-1)\vones$, and with $I$ defined to be the ideal generated by all monomials $f$ of $\cI$ with $f \leq \a$ (in other words, throw out all generators of the form $x_i^n$ for $1 \leq i \leq d$, all of which show up in $\cI$ regardless of the arrangement $\A$).  We then have from Corollary \ref{cor:cellres} that $\Sigma_{\A}$ supports a minimal cellular resolution of $\cI = I + \m^{\a + 1}$.  The conditions of Theorem \ref{resduality} are met and we conclude that $(\overline \Sigma_{\A})_{\leq \a}$ supports a minimal cocellular resolution of $I^{[\a]}$.

We next determine $I^{[\a]}$, the Alexander dual of $I$ with respect to $\a = (n-1)\vones$.  In \cite{MilStu05} it is shown that if $\b \leq \a$ then ${\bf x^b}$ lies outside $I$ if and only if ${\bf x^{a-b}}$ lies inside $I^{[\a]}$.  Hence to find a set of generators for $I^{[\a]}$ it suffices to determine the maximal monomials which lie outside $I$.  But these monomials correspond to the minimal cotypes that arise in the complex $\CD_{\A}$, and these are given by the collection of monomials $\overline \ct(x)$, for $x$ a 0-dimensional cell in $\CD_{\A}$.  Hence we conclude that $I^{[\a]} = \ccI$.

For the second part of the claim, we note that a face $\sigma \in \overline
\Sigma_{\A}$ with label $\b$ satisfies $\b \leq \a$ if and only if $n \vones -
\ct(\sigma) \leq (n-1) \vones$ for the coarse type label $\ct(\sigma)$ in
$\Sigma_{\A}$.  But this occurs exactly when $\ct(\sigma) \geq \vones$, which
happens if and only if the face $\sigma$ is not contained in the boundary of
$\Sigma_{\A}$.  Hence the cocellular resolution of $I^{[\a]}$ that we obtain,
supported on $(\overline \Sigma_{\A})_{\leq \a}$, is given by the relative
cocellular complex of $(\Sigma_{\A}, \partial \Sigma_{\A})$.  By duality, the
relative cochain complex $C^*(\Sigma_{\A}, \partial \Sigma_{\A})$ is
isomorphic to the chain complex $C_*(\boundedCD_{\A})$ of the tropical complex $\boundedCD_{\A}$ (the bounded subcomplex of the decomposition of ${\T}^{d-1}$ induced by $\A$).  Hence $\boundedCD_{\A}$, with labels given by coarse \emph{cotype}, supports a minimal \emph{cellular} resolution of the ideal $I^{[\a]} = \ccI$.
\end{proof}

\section{Examples}\label{sec:examples}

\noindent
In this section we discuss some examples of our constructions and results.  We begin with
a family of generic arrangements which correspond to well-known objects in the theory of
subdivisions (staircase triangulations) and tropical convexity (cyclic polytopes).  In the
context of the associated coarse type ideal we recover a construction of a cellular
resolution first described by Sinefakopoulos in \cite{Sine08}.  We then discuss a family
of non-generic examples which arise from tropical hypersimplices; in this we are able to
explicitly describe the generators of the coarse type ideal.

\subsection{The staircase triangulation}
One particularly well-behaved fine mixed subdivision of $n \Delta_{d-1}$ comes from
applying the Cayley trick (discussed in Section \ref{sec:Cayley} above) to the so-called
\emph{staircase} triangulation of $\Delta_{n-1} \times \Delta_{d-1}$.  The origins of the
staircase triangulation date back at least to work of Eilenberg and Zilber in
\cite{EZ53}, where the authors introduce algebraic operations which they call `shuffles' to
study the homology groups of products of spaces.  We wish to recall the construction here
in language suitable for our purposes.

As above, we use $\{e_1, \dots, e_d\}$ to denote the vertices of the simplex
$\Delta_{d-1}$ and consider fine mixed cells of the following kind.  Given a sequence of
integers $(b_1, b_2, \dots, b_{n+1})$ satisfying $1 = b_1 \leq b_2 \leq \cdots \leq b_n
\leq b_{n+1} = d$ we let $B_i := \{e_{b_i},e_{b_i+1},\dots,e_{b_{i+1}}\}$ for $1 \leq i
\leq n$, and use $(b_1,b_2,\dots,b_{n+1})$ to denote the corresponding (fine) mixed cell
$B_1 + B_2 + \cdots + B_n$ of $n \Delta_{d-1}$.  Note that each such sequence can be
thought of as a `staircase' of length $n$ and height $d$, where at the $i$-th step one
should climb a step of height $b_{i+1} - b_i$.

We claim that the collection of fine mixed cells corresponding to
$(b_1,b_2,\dots,b_{n+1})$ for $ 1
= b_1 \leq b_2 \leq \cdots \leq b_{n} \leq b_{n+1} = d$ forms a fine mixed
subdivision of the complex $n \Delta_{d-1}$.  To see this, we once again
employ the Cayley trick and consider the (proposed) triangulation of the
product $\Delta_{n-1} \times \Delta_{d-1}$.  As discussed above, the vertices
of $\Delta_{n-1} \times \Delta_{d-1}$ are pairs $(e_i, e_j)$, where $1 \leq i
\leq n$ and $1 \leq j \leq d$.  A simplex in any triangulation of
$\Delta_{n-1} \times \Delta_{d-1}$ corresponds to a tree in the complete
bipartite graph $K_{n,d}$, where $(v_i, w_j)$ is an edge of the graph whenever
$(e_i, e_j)$ is a vertex of the simplex.  The corresponding mixed cell $B_1 +
\cdots + B_n$ of $n \Delta_{d-1}$ has as the $i$-th summand the face
$\{e_j:w_j \in N(v_i)\}$, given by all vertices that are adjacent to $v_i$.
Note that the coarse type of the mixed cell is given by the degree sequence of
the $d$ vertices in the second vertex set partition.

Working with $\Delta_{n-1} \times \Delta_{d-1}$ has the advantage that one can readily
verify if a collection of simplices does, in fact, give a triangulation of the space, in
terms of the corresponding subgraphs.  We omit the details here, but point out that the
staircase triangulation also arises from the so-called `cyclic arrangement' of $n$
tropical hyperplanes in $\T^{d-1}$ as discussed in \cite{BlockYu06}.  We call the
associated fine mixed subdivision of $n \Delta_{d-1}$ the \emph{staircase mixed
  subdivision}.  As with any fine mixed subdivision, the 0-dimensional cells of the
staircase mixed subdivision of $n \Delta_{n-1}$ are labeled by the monomial generators of
$\langle x_1, \dots, x_d \rangle ^n$, and from Corollary \ref{cor:cellres_generic} we know
that this labeled complex supports a minimal cellular resolution of $\langle x_1, \dots,
x_d \rangle ^n$.  In fact, the case of $n=d=3$ is hinted at in
\cite[Example~2.20]{MilStu05}.
% Michael: Here and elsewhere I replaced "a" (minimal resolution) by "the", because it is unique.
% Anton: I would use "a" *cellular* minimal resolution since I don't know what isomorphism of *cellular* resolutions means.
\begin{rem}
  In \cite{Sine08}, Sinefakopoulos constructs a labeled polyhedral complex which he calls
  $P_n(x_1, \dots, x_d)$, and shows that it supports the minimal resolution of $\langle x_1,
  \dots, x_d \rangle ^n$.  The complex $P_n(x_1, \dots, x_d)$ is constructed inductively,
  with each $P_n(x_{k+1}, \dots, x_d)$ a subcomplex of $P_n(x_k, \dots, x_d)$ for all $k <
  d$.  In \cite{Sine08} it is shown that $P_n(x_1, \dots, x_d)$ can in fact be realized as
  a subdivision of the dilated simplex $n \Delta_{d-1}$, and here we claim that this
  complex is isomorphic (as a labeled complex) to the staircase mixed subdivision.

  Mimicking the construction of $P_n(x_1, \dots, x_d)$ from \cite{Sine08}, we proceed by
  induction on $n$.  If $n=1$, both $P_1(x_1, \dots, x_d)$ and the staircase mixed
  subdivision correspond to the standard $(d-1)$-dimensional simplex with vertex labels
  $\{x_1, \dots, x_d\}$, denoted $\Delta_{d-1}(x_1, \dots, x_d)$.  We assume that for $n
  \geq 1$, the complex $P_n(x_1, \dots, x_d)$ is isomorphic as a labeled complex to the
  staircase subdivision of $n \Delta_{d-1}$, with this isomorphism restricting to an
  isomorphism on each $P_n(x_k, \dots, x_d)$.  In \cite{Sine08} the inductive step of the
  construction is obtained as
  \[
  P_{n+1}(x_1, \dots, x_d) \ := \ C_1 \cup \cdots \cup C_d \, ,
  \]
  where each $C_k$ is defined as
  \[C_k \ := \ \Delta_{k-1}(x_1, \dots, x_k) \times P_n(x_k, \dots x_d) \, .
  \]

  For any $k$, we claim that the complex $C_k$ corresponds to the complex composed of the
  set of fine mixed cells $(b_1,b_2, \dots, b_{n+1})$ with $b_2 = k$.  This would
  establish the desired isomorphism since the set of fine mixed cells in the staircase
  mixed subdivision correspond to the all such sequences with $1 \leq b_2 \leq d$.  To
  prove the claim, we note that for any labeled complex $P$ the product $\Delta_{k-1}(x_1,
  \dots, x_k) \times P$ is isomorphic to $\Delta_{k-1}(x_1, \dots, x_k) + P$, assuming
  that these complexes lie in affinely independent subspaces.  Furthermore, the vertices
  of the product complex are labeled by sums of vertices from each summand.  Now, by
  induction we have that each $P_n(x_k, \dots, x_d)$ is isomorphic to the corresponding
  subcomplex of the staircase mixed subdivision of $n \Delta_{d-1}$.  In particular, the
  set of mixed cells correspond to the set of sequences $(b_1, b_2, \dots, b_{n+1})$ with
  $k = b_1 \leq b_2 \cdots \leq b_n \leq b_{n+1} = d$.  The facets of $C_k$ are obtained
  by taking the Minkowski sum of the simplex $\Delta_{k-1}(x_1, \dots, x_k)$ with these
  mixed cells, and hence $C_k$ is isomorphic to the complex composed of all fine mixed
  cells $(b_1,b_2, \dots, b_{n+1})$ with $b_2 = k$.
\end{rem}

Hence our constructions recover the result from \cite{Sine08} (with a more explicit
description of the underlying polyhedral complex).  Once again, we wish to emphasize that
from Corollary \ref{cor:cellres_generic} we know that in fact \emph{any} (regular) fine
mixed subdivision of $n \Delta_{d-1}$ supports a minimal cellular resolution of the
homogeneous ideal $\langle x_1, \dots, x_d \rangle^n$.

\subsection{Tropical hypersimplices}

The hypersimplex $\Delta(k,n)$ is an ordinary convex polytope which, e.g.,
naturally turns up in the study of tropical Grassmannians of tropical $k$-planes in
$\T^{n-1}$.  Its tropical counterpart, the \emph{tropical hypersimplex}
$\Delta^{\trop}(k,n)$ is defined as the tropical convex hull of all $0/1$-vectors of
length $n$ with exactly $k$ zeros.  Notice that we have the strict inclusions
\[
\Delta^{\trop}(1,n) \ \supsetneq \ \Delta^{\trop}(2,n) \ \supsetneq \ \cdots \
\supsetneq \Delta^{\trop}(n-1,n)
\]
of subsets of $\T^{n-1}$.  We want to determine the coarse type ideal corresponding to the
configuration of $\tbinom{n}{k}$ points in $\T^{n-1}$ given by the tropical vertices of
$\Delta^{\trop}(k,n)$.  The $n$ generators of $\Delta^{\trop}(1,n)$ are in general
position.  Hence the coarse type ideal is the homogeneous maximal ideal $\langle
x_1,x_2,\dots,x_n\rangle$ in this case.  The second tropical hypersimplex
$\Delta^{\trop}(2,n)$ is contained in the $\min$-tropical hyperplane with the origin as
its apex.  In particular, $\Delta^{\trop}(2,n)$, seen as a polytopal complex in
$\R^{n-1}=\T^{n-1}$, is of dimension $n-2$.  This implies that all maximal cells in the
type decomposition of $\T^{n-1}$ induced by the tropical vertices of $\Delta^{\trop}(k,n)$
are unbounded if $2\le k<n$.  Equivalently, the coarse type ideal only has minimal
generators with the property that at least one variable is missing (that is, the exponent of
this variable is zero).  The symmetric group $\Sym(n)$ acts on the set of tropical
vertices of $\Delta^{\trop}(k,n)$, and hence it also acts on the maximal cells of the type
decomposition.

\begin{prop}\label{prop:hypersimplex}
  Let $2\le k <n$. Then up to the $\Sym(n)$ action the coarse types of the maximal cells
  in the type decomposition of $\T^{n-1}$ induced by the $\tbinom{n}{k}$ tropical vertices
  of $\Delta^{\trop}(k,n)$ are given by
  \[
  (\binom{n-\alpha}{k}+\binom{n-1}{k-1},\ \binom{n-2}{k-1},\dots,\binom{n-\alpha}{k-1},\ \underbrace{0,\dots,0}_{n-\alpha}) \,
  \]
  where $1 \le \alpha \le n-k+1$.
\end{prop}

This result has been obtained by Katja Kulas, and a complete proof will appear in
\cite{Kulas}.  Here we only sketch the argument.  Let $\A$ denote the arrangement of
$\max$-tropical hyperplanes induced by the tropical vertices of $\Delta^{\trop}(k,n)$.
The maximal dimension of a bounded cell or, equivalently, the tropical rank of the matrix
whose columns are the tropical vertices of $\Delta^{\trop}(k,n)$, equals $n-k$
\cite[Proposition~7.2]{DevelinSantosSturmfels05}. All cells in the arrangement $\A$,
bounded or not, are convex polyhedra in $\R^{\tbinom{n}{k}-1}$ which are \emph{pointed},
that is, they do not contain any affine line.  This implies that each cell must contain a
bounded cell as a face.  Let $C$ be some maximal cell in the arrangement $\A$.  From now
on we assume that $k>1$ whence $C$ is necessarily unbounded.  Let $\alpha-1$ be the
maximal dimension of a bounded cell in the boundary of $C$; and we call $\alpha$ the
\emph{class} of $C$.  From the bound on the tropical rank it follows that $1 \le \alpha
\le n-k+1$.  All these values for $\alpha$ actually occur.  It turns out that $\Sym(n)$
acts transitively on the set of maximal cells of class $\alpha$.  The coarse type of one
representative of each class is given in Proposition~\ref{prop:hypersimplex}.  Due to
Corollary~\ref{cor:coarse_min_gen} this yields generators for the corresponding coarse
type ideal.

\section{Face counting and incidence structure\\
  of the tropical complex}\label{sec:facecounting}

\noindent
In Section \ref{sec:resolutions} we saw how the polyhedral complexes $\CD_\A$ and
$\boundedCD_\A$ associated to an arrangement $\A$ gave rise to resolutions of the coarse
type ideal $\cI$ and the fine cotype ideal $\fcI$, respectively.  The \emph{minimality} of
our resolution also leads to some important implications regarding the combinatorics of
$\CD_\A$ and $\boundedCD_\A$ themselves.  In this section we discuss face numbers of
tropical complexes, as well as an algorithm for determining the facial structure
of $\boundedCD_\A$ given the arrangement $\A = \A(V)$.  The latter generalizes a result of
\cite{BlockYu06}, where a similar algorithm for the case of sufficiently generic arrangements was provided.

\subsection{Counting faces}

As a first application we point out that the $f$-vector of $\CD_\A$ can be determined from
the $\Z$-graded (`coarse') Betti numbers of $\cI$.  We noted in
Proposition~\ref{prop:bounded_coarse} that from the coarse type it is possible to
distinguish bounded from unbounded cells in $\CD_{\A}$.  Thus, we can also recover the
numerical behavior of the bounded complex $\boundedCD_{\A}$.

\begin{cor}\label{cor:face_numbers}
  Let $\A = \A(V)$ be a tropical hyperplane arrangement in $\T^{d-1}$ and let $\cI$ be its
  coarse type ideal. Then the number of cells in $\CD_{\A}$ of dimension $k$ is \[
  f_{k}(\CD_{\A}) \ =\ \beta_{d-1-k}(\cI) \ =\ \sum_{\b \in \Z^d} \beta_{d-1-k,\b}(\cI) \,
  .  \] The number of bounded $k$-cells in $\CD_{\A}$ is given by as the sum of Betti
  numbers $\beta_{d-1-k,\b}(\cI)$ for which $\b > 0$.
\end{cor}

Note that for the $k$-cells we need to consider the $(d{-}1{-}k)$-th Betti numbers.  This
is due to the fact that $\CD_{\A}$ supports a \emph{co}cellular resolution.

Furthermore, we can use the \emph{uniqueness} of minimal resolutions to derive further
results regarding face numbers of arrangements.  For this, suppose $\A$ is a sufficiently
generic arrangement of $n$ tropical hyperplanes in $\T^{d-1}$, and let $\CD_{\A}$ denote
the induced polyhedral subdivision of $\T^{d-1}$ determined by type.  In Corollary
\ref{cor:cocellular_generic} we saw that $\CD_{\A}$, with labels given by coarse type, supports a minimal cocellular
resolution of the ideal $\ideal{x_1,\dots,x_d}^n$.  Any two such resolutions are
isomorphic as chain complexes, and in particular the finely graded Betti numbers
$\beta_{i,\sigma}$ do not depend on the resolution.  By construction, $\beta_{i,\sigma}$
is precisely the number of cells in $\CD_{\A}$ with monomial label $\sigma$.  But the
monomial labels are given by the coarse types, and hence we obtain the following.

\begin{cor} \label{cor:coarsetypes_inv}
    Let $\A$ be a sufficiently generic arrangement of $n$ hyperplanes in $\T^{d-1}$. For
    every $0 \le i \le d-1$ the collection of coarse types $\ct(C_T)$ for
    $\dim C_T = i$, counted with multiplicities, is independent of the
    arrangement.
\end{cor}

Putting the above result in perspective with the second statement of
Corollary~\ref{cor:face_numbers}, this proves the following result without appealing to
the equidecomposability of the product of simplices.

\begin{cor}\label{cor:generic_f_vector}
    The number of cells of $\CD_{\A}$ for a tropical hyperplane arrangement
    $\A$ in general position is independent of the choice of hyperplanes.
    More precisely, the number of $k$-dimensional cells induced by the
    arrangement of $n$ tropical hyperplanes in $\T^{d-1}$ equals
    \[
        f_k(\CD_A) \ = \ \sum_{\ell = 0}^k
            \binom{n + d - 2 - \ell}{n - 1}
            \binom{d - 1 - \ell}{d - 1 - k}
         \, .
    \]
    %the number of compositions of $n+k$ into $d$ parts such that no part
    %exceeds $n$.
\end{cor}
\begin{proof}
  By Corollaries~\ref{cor:face_numbers} and~\ref{cor:cocellular_generic} we have
  $f_i(\CD_{\A}) = \beta_{d-1-i}(\ideal{x_1,\dots,x_n}^n)$.  The Betti numbers of the
  power of the homogeneous maximal ideal are well-known.  For example, they can be
  determined as follows.

  The ideal $\m^n = \ideal{x_1,\dots,x_d}^n$ is \emph{strongly stable}, that is,
  $\frac{x_i}{x_j} \x^\b \in \m^n$ for every $\x^\b$ monomial divisible by $x_j$ and $i <
  j$.  In particular, $\m^n$ is Borel-fixed, and hence the Betti numbers are given
  by~\cite[Thm.~2.18]{MilStu05}, and we have
  \[
  \beta_i(\m^n) \ = \ \sum_{\a \in \N^d, |a| = n} \binom{ \max(\x^\a) - 1}{i}
  \]
  where $\max(\x^\a) = \max \{ i : \a_i > 0 \}$.  Now if $\max(\x^\a) =
  \ell$, then
  \[
  \a \ = \ (a_1,a_2,\dots,a_\ell + 1, 0,\dots,0)
  \]
  with $a_1,\dots,a_\ell \ge 0$ and $\sum_i a_i = n-1$. Hence, the number of generators
  $\x^\a$ with $\max(\x^a) = \ell$ is the number of monomials in $\ell$ variables of total
  degree $n-1$. This yields
  \[
  \beta_i(\m^n) \ = \ \sum_{\ell = 1}^d \binom{n-2+\ell}{n-1}\binom{\ell - 1}{i}
  \]
  Now substitute $i$ in the formula with $d-1-k$.
\end{proof}

In the context of mixed subdivisions of dilated simplices, the previous two corollaries yield the following.

\begin{cor}\label{cor:genericmixedcells}
  In any regular fine mixed subdivision $\Sigma$ of $n \Delta_{d-1}$ the collection of coarse
  types, counted with multiplicities, is independent of the subdivision.  In particular, $f_i(\Sigma)$, the number of $i$-dimensional faces of $\Sigma$ is independent of the subdivision, and is given by
    \[
        f_i(\Sigma) = \sum_{\ell = 1}^d \binom{n-2+\ell}{n-1}\binom{\ell -
        1}{i}.
    \]
\end{cor}

\begin{rem}
  For the case of vertices, Corollary \ref{cor:genericmixedcells} was already established
  in \ref{prop:latticepoints}, where the it was shown that the $0$-dimensional cells in
  any fine mixed subdivision correspond to the lattice points of $n \Delta_{d-1}$.  The
  case of facets can also be proven directly using the following combination of mixed
  volume calculations and the switch duality of tropical complexes.  Given a
  fine mixed subdivision $\Sigma$ of $n \Delta_{d-1}$ we consider the corresponding fine
  mixed subdivision $\Sigma^*$ of $d \Delta_{n-1}$ induced by applying the Cayley trick
  and switching the factors in $\Delta_{n-1} \times \Delta_{d-1}$.  For nonnegative real
  numbers $\lambda_1, \dots, \lambda_d$, we have that the linear functional given by the
  volume $\vol(\lambda_1 \Delta_{n-1} + \cdots + \lambda_d \Delta_{n-1})$ can be expressed
  as a homogeneous polynomial of degree $n-1$ in the variables $\lambda_1, \dots,
  \lambda_d$.  The coefficient of the monomial $\lambda_1^{i_1} \lambda_2^{i_2} \cdots
  \lambda_d^{i_d}$ is called a mixed volume, and for any fine mixed subdivision of $d
  \Delta_{n-1}$ it is equal to the sum of the volumes of the mixed cells of the
  corresponding \emph{dual} coarse type.  In our case all fine mixed cells of a certain
  dual coarse type have the same volume, since the relevant data is given by the number of
  vertices in each factor, regardless of the labels on the vertices.  This implies that in
  any fine mixed subdivision the number of fine mixed cells with a certain dual coarse
  type is always the same.  But the collection of \emph{dual coarse} types for this fine
  mixed subdivision coincides with collection of \emph{coarse} types for the given fine
  mixed subdivision of $n \Delta_{d-1}$, proving our claim.

  In fact (for $k=0$), each coarse type shows up at most once and the collection of coarse
  types can be seen to coincide with the monomial generators of the ideal
  \[
  x_1 x_2 \cdots x_d \langle x_1, x_2, \dots, x_d \rangle^{n-1} \, .
  \]
  \noindent
  This follows from the fact that in this case the volume polynomial of the dual mixed
  subdivision is given by $(\lambda_1 + \cdots + \lambda_d)^{n-1}$ and the volume of a
  mixed cell of a particular dual coarse type is exactly the coefficient of the
  corresponding monomial.  We do not know of a similar mixed volume interpretation of the other face numbers.
\end{rem}

\subsection{Incidence structure of the bounded complex from the fine cotype ideal}

In \cite{BlockYu06} Block and Yu develop an algorithm to compute the bounded
complex of a generic tropical hyperplane arrangement $\A = \A(V)$ that draws
from computational commutative algebra.  There it is shown that the
bounded complex supports a minimal cellular resolution of a monomial ideal which we have called the fine cotype ideal $I_{\fct(\A)}$ associated to $\A$.  It turns out that this ideal is an initial ideal of the toric ideal $I_{n,d}$
%\ideal{x_{ik}x_{jl} - x_{il}x_{jk}: i,j \in [n], k,l \in [d] }$
for the vertices of $\Delta_{n-1} \times \Delta_{d-1}$.  The results in
\cite{BlockYu06} rely on the genericity of the arrangement in two ways: i) the
fact that the bounded complex supports a resolution of the fine cotype ideal
is proved by appealing to the \emph{polyhedral detour} to tropical convexity
presented in \cite{DevStu04}, and ii) the genericity is needed to guarantee
that the initial ideal is indeed monomial.

In this section we extend their algorithm to the case of non-generic
hyperplane arrangements. We bypass the two mentioned dependences on genericity
as follows.  In Section \ref{sec:resolutions} we have already shown that the bounded complex resolves the fine cotype ideal for an \emph{arbitrary} arrangement, using first principles in tropical convexity. As for ii), we next show that the fine
cotype ideal is the maximal monomial ideal contained in the initial ideal
associated to the weights $V$. In terms of polyhedral geometry, this
corresponds to the passage from the polyhedral subdivision $\Sigma_V$ of
$\Delta_{n-1} \times \Delta_{d-1}$ induced by $V$ to the \emph{crosscut
complex}.

The results of the previous sections cannot directly be turned into a practical algorithm
as the fine cotype ideal can only be determined \emph{after} computing the bounded complex
or, at least, the vertices of $\CD_\A$.  In the case of a generic arrangement of tropical
hyperplanes this problem is resolved in \cite{BlockYu06} as follows.  The ideal
\[
I_{n,d} \ = \ \ideal{x_{ik}x_{jl} - x_{il}x_{jk}: i,j \in [n], k,l \in [d] }
\]
of $2\times2$-minors of a general $n\times d$-matrix in $\field[x_{11},\dots,x_{nd}]$ is
the toric ideal associated to the vertices of the ordinary lattice polytope $\Delta_{n-1}
\times \Delta_{d-1}$. The initial ideal $\iform_V(I_{n,d})$ with respect to the weights
$V$ is a squarefree monomial ideal and, by a celebrated result of Sturmfels
\cite[Thm.~8.3]{stu96}, equal to the Stanley-Reisner ideal $I_\Sigma$ of the triangulation
$\Sigma$ of $\Delta_{n-1} \times \Delta_{d-1}$ induced by $V$; see
Section~\ref{sec:Mixed}.  The face poset of the bounded complex $\boundedCD_\A$ is
anti-isomorphic to the subposet of \emph{interior} cells of $\Sigma$. In short, there is a
bijection between $i$-cells of $\CD_\A$ and $(n+d-2-i)$-cells of $\Sigma$ which lie in the
interior of $\Delta_{n-1}\times\Delta_{d-1}$. The bijection takes a cell $C_T^\circ$ of
fine type $T$ to the cell of $\Sigma$ with vertices $(e_i,e_j)$ for $T_{ij} = 1$. Finally,
the Alexander dual of $I_\Sigma$ is the monomial ideal generated by the fine cotypes of
the vertices of $\CD_\A$.

If the placement $V$ of the hyperplanes is not generic, then the induced subdivision
$\Sigma$ is not a triangulation. However, the above bijection is unaffected and, in
particular, the fine type of a vertex of $\CD_\A$ can be read off the corresponding facet
of $\Sigma$. We define the \emph{crosscut complex} $\ccut{\Sigma} \subseteq 2^{[n] \times
  [d]}$ to be the unique simplicial complex with the same vertices-in-facets incidences as
the polyhedral complex $\Sigma$.  The crosscut complex is a standard notion in
combinatorial topology and can be defined in more generality (see Bj\"{o}rner
\cite[pp.~1850ff]{bjorner95}).  The following observation is immediate from the
definitions.

\begin{prop}
  For $V$ an ordered sequence of $n$ points in $\T^{d-1}$ let $\A = \A(V)$ be the
  corresponding tropical arrangement and let $\Sigma$ be the regular subdivision of
  $\Delta_{n-1} \times \Delta_{d-1}$ induced by $V$. Then the fine cotype ideal $\fcI$ is
  Alexander dual to the Stanley-Reisner ideal of $\ccut{\Sigma}$.
\end{prop}

The crosscut complex encodes the information of which collections
of vertices lie in a common face. Hence, the crosscut complex is a
purely combinatorial object and does not see the affine structure of the
underlying polyhedral complex.

Algebraically, the initial ideal $\iform_V(I_{n,d})$ is a coherent $A$-graded
ideal in the sense of \cite[Ch.~10]{stu96} and encodes the corresponding
polyhedral subdivision $\Sigma$ induced by $V$ (see \cite[Thm.~10.10]{stu96}).
The ideal $\iform_V(I)$ of a toric ideal $I$ is generated by monomials and
binomials. Intuitively, the binomials encode the affine structure within cells
of $\Sigma$, that is, the affine dependencies, while the monomial generators
encode the Stanley-Reisner data for the crosscut complex. For an arbitrary
ideal $J$, denote by $M(J)$ the largest monomial ideal contained in $J$.
% Michael: "dependency" does exist as a word (which is why it passes spell checking) but
% it means something different; try dict.leo.org

\begin{lem}
    Let $I_A$ be the toric ideal for $A \in \N^{d \times n}$ and for $\omega
    \in \R^n$ let $J = \iform_\omega(I_A)$ and $\Sigma$ the regular
    subdivision of $A$ induced by $\omega$. Then the radical of $M(J)$ is the
    Stanley-Reisner ideal of the crosscut complex of $\Sigma_\omega$, that is~
    \[
        I_{\Rad(M(J))} \ = \ I_{\ccut{\Sigma}} \, .
    \]
\end{lem}
\begin{proof}
    By Theorem~10.10 of \cite{stu96}, the ideal $J$ is the intersection of
    ideals $J_\sigma$ indexed by the faces $\sigma \in \Sigma$ and
    $J_\sigma$ is torus isomorphic to $I_\sigma = I_A + \ideal{ x_i : i
    \not\in \sigma}$. Hence, we have
    \[
        M(J) \ = \ \bigcap_{\sigma \in \Sigma} M(I_\sigma).
    \]
    Under the projection map $\field[x_1,\dots,x_n] \rightarrow \field[x_i : i
    \in \sigma]$ that takes $x_i \mapsto 0$ for $i \not\in \sigma$, $I_\sigma$
    is isomorphic to the toric ideal corresponding to the columns of $A$
    indexed by $\sigma \subseteq [n]$. Hence, a monomial $\x^\a$ is contained
    in $I_\sigma$ if and only if $\mathrm{supp}(\x^a) \not\subseteq \sigma$.
    Therefore $M(I_\sigma) = \ideal{ x_i : i \not\in\sigma}$ and for $\tau
    \subseteq [n]$ we have that $\x^\tau \in \Rad(M(J))$ if and only if $\tau$
    is not contained in any cell of~$\Sigma$.
\end{proof}

A special case of this lemma appears in \cite[Lem.~4.5.4]{sst00}.
Since the toric ideal $I_{n,d}$ is unimodular, the ideal
$M(\iform_V(I_{n,d}))$ is automatically squarefree and we obtain the following
corollary.

\begin{cor} \label{cor:bounded_compute} Let $\A = \A(V)$ be a tropical hyperplane
  arrangement and let $J = \iform_V(I_{n,d})$ the initial ideal of $I_{n,d}$ for the
  weights $V$.  Then the Alexander dual of the squarefree monomial ideal $M(J)$ is the
  fine cotype ideal of $\A$.
\end{cor}

In the case that $\A$ is generic, the ideal $M(J)$ coincides with the initial ideal
$\iform_V(I_{n,d})$ and hence recovers the main result of \cite{BlockYu06}. Moreover, it
entails the modification of the algorithm in \cite{BlockYu06} by replacing the ideal
$\iform_V(I_{n,d})$.  We describe our Algorithm~\ref{algo:poset} below.  For an overview
of other algorithms to produce the same output via techniques from polyhedral
combinatorics and algorithmic geometry see \cite{Joswig09}.

\begin{algorithm}[H]
  \caption{Computing the face poset of the tropical complex}
  \label{algo:poset}
  \dontprintsemicolon

  \Input{matrix $V\in\R^{n\times d}$}
  \Output{face poset of $\boundedCD_{\A(V)}$}

  calculate the initial ideal $J = \iform_V(I_{n,d})$ \;
  calculate $M(J) = \ideal{ \x^\a : \x^\a \in J }$ \;
  calculate the Alexander dual $M(J)^*$ \;
  find a minimal fine graded resolution to determine the face poset
  of $\boundedCD_{\A(V)}$ \;
\end{algorithm}

An algorithm for calculating $M( \cdot )$ for a toric ideal is given in
\cite[Algo.~4.4.2]{sst00} in terms of elimination theory.  Computing the Alexander dual of
a finite simplicial complex can be accomplished via an algorithm of Lawler
\cite{Lawler66}.  In \cite{LaScalaStillman98} La Scala and Stillman give an overview of
methods to compute minimal resolutions.  To illustrate the results in this section, we
show an example computed with the computer algebra system {\tt Macaulay2}~\cite{M2}.

\begin{figure}[htb]
  \includegraphics[scale=1]{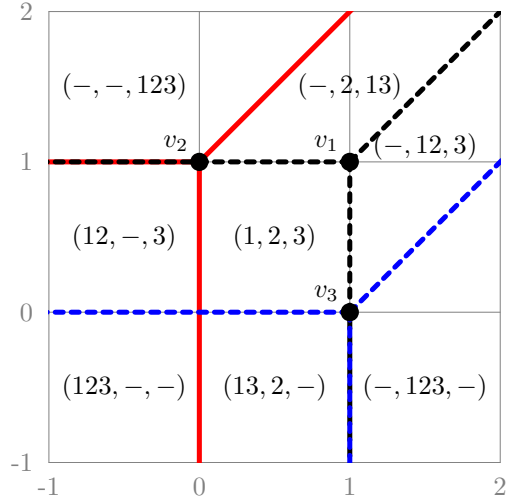}
  \caption{Non-generic arrangement in $\T^2$ with labeling by fine type.}
  \label{fig:non_generic}
\end{figure}

\begin{example}
  Consider the points $v_1=(0,1,1)$, $v_2=(0,0,1)$, and $v_3=(0,1,0)$ in $\T^2$.  The
  induced type decomposition is shown in Figure~\ref{fig:non_generic}.  We have
  \begin{align*}
    I_{3,3} \ = \ \langle&
    x_{11}x_{22}-\underline{x_{12}x_{21}},\
    x_{11}x_{33}-\underline{x_{13}x_{31}},\
    \underline{x_{12}x_{31}-x_{11}x_{32}},\\
    &x_{12}x_{33}-\underline{x_{13}x_{32}},\
    \underline{x_{13}x_{21}-x_{11}x_{23}},\
    x_{13}x_{22}-\underline{x_{12}x_{23}},\\
    &x_{21}x_{33}-\underline{x_{23}x_{31}},\
    x_{22}x_{31}-\underline{x_{21}x_{32}},\
    x_{22}x_{33}-\underline{x_{23}x_{32}}
    \rangle \, ,
  \end{align*}
  where the initial forms of the generators with respect to the weight matrix
  $V=\left(\begin{smallmatrix}v_1\\ v_2\\ v_3\end{smallmatrix}\right)\in\R^{3\times 3}$
  are underlined; these generate the toric ideal $J$.  Its maximal monomial subideal
  equals
  \[
  M(J) \ = \ \langle
  x_{12}x_{21},\ x_{12}x_{23},\ x_{13}x_{31},\ x_{23}x_{31},\ x_{13}x_{32},\
  x_{21}x_{32},\ x_{23}x_{32}
  \rangle \, ,
  \]
  and this is squarefree.  Its Alexander dual is
  \[
  M(J)^* \ = \ \langle
  x_{13}x_{21}x_{23},\ x_{12}x_{13}x_{23}x_{32},\ x_{12}x_{31}x_{32},\
  x_{21}x_{23}x_{31}x_{32} \rangle \, .
  \]
  These four generators of $M(J)^*$ encode the fine cotypes of the points $v_3$, $(0,0,0)$,
  $v_2$, and $v_1$, respectively.  Setting $S=\field[x_{11},x_{12},\dots,x_{33}]$ we
  obtain the minimal free resolution
  \[
  0 \to S \stackrel{\phi_3}{\longrightarrow} S^4 \stackrel{\phi_2}{\longrightarrow} S^4 \stackrel{\phi_1}{\longrightarrow} I
  \to 0 \, ,
  \]
  where the non-trivial differentials $\phi_i$ are given by the matrices
  \begin{align*}
    \phi_1 \ &= \
    \begin{pmatrix}
      x_{13}x_{21}x_{23} & x_{12}x_{31}x_{32} & x_{21}x_{23}x_{31}x_{32} &
      x_{12}x_{13}x_{23}x_{32}
    \end{pmatrix} \\
    \phi_2 \ &= \
    \begin{pmatrix}
      0             & -x_{31}x_{32}  & -x_{12}x_{32} & 0               \\
      -x_{21}x_{23}   & 0            & 0           & -x_{13}x_{23} \\
      x_{12}         & x_{13}        & 0           & 0               \\
      0             & 0            & x_{21}       & x_{31}
    \end{pmatrix} \\
    \phi_3 \ &= \
    \begin{pmatrix}
      -x_{13}\\
      x_{12} \\
      -x_{31} \\
      x_{21}
    \end{pmatrix} \, .
  \end{align*}
  These matrices are to be multiplied to column vectors from the left.  The non-zero finely graded Betti numbers are
  \begin{align*}
    &\beta_{0,(12,13,23,32)} = \beta_{0,(12,31,32)} = \beta_{0,(21,23,31,32)} = \beta_{0,(13,21,23)} = 1 \\
    &\beta_{1,(12,13,23,31,32)} =  \beta_{1,(12,21,23,31,32)} = \beta_{1,(13,21,23,31,32)} = \beta_{1,(12,13,21,23,32)} = 1 \\
    &\beta_{2,(12,13,21,23,31,32)} = 1 \, ,
  \end{align*}
  where, for example, $(12,13,23,32)$ is the squarefree monomial
  $x_{12}x_{13}x_{23}x_{32}$ corresponding to the point $(0,0,0)$.  Note that we are
  resolving the ideal $M(J)^*$ rather than the quotient $S/M(J)^*$ (which would yield a
  shift of +1 in the first coordinate of each Betti number).  The non-zero coarsely graded
  Betti numbers are then
  \[
  \beta_{0,3} = \beta_{0,4} = 2, \ \beta_{1,5} = 4, \ \beta_{2,6} = 1
  \, .
  \]
\end{example}

\section{Further remarks and open questions}\label{sec:final}

\noindent
Having constructed cellular resolutions of ideals arising from regular mixed subdivisions
of dilated simplices, a natural question to ask is if the assumption of regularity is
really necessary.  The relevant properties of our subdivisions were established by
considering them as induced by arrangements of tropical hyperplanes, and hence these
subdivisions were always regular.  However, the construction of a labeled complex from an
arbitrary mixed subdivision of $n\Delta_{d-1}$ still makes sense, and it is an open
question (as far as we know) whether these also support cellular resolutions.  As a
special case, in light of Proposition \ref{prop:latticepoints} we can ask whether \emph{any} fine mixed subdivision of $n \Delta_{d-1}$ supports a minimal cellular resolution of $\langle x_1, \dots, x_d \rangle^n$.

A connection to tropical geometry is provided by the \emph{tropical oriented matroids} of
Ardila and Develin from~\cite{AD09}. There the authors introduce an axiomatic approach to
the study of (fine) types, with a list of properties which they show are satisfied by the
collection of fine types arising from an arrangement of tropical hyperplanes.  It is
conjectured that all abstract oriented matroids are realized by arbitrary subdivisions,
and if this were the case we might think of the ideals described in the previous paragraph
as `tropical oriented matroid ideals'.  A further task would be to relate the algebraic
properties of these ideals with the combinatorial properties of the underlying matroid, in
the spirit of \cite{NovPosStu}.

As mentioned above, another unresolved question is to characterize which monomial ideals
arise as coarse type ideals for some tropical hyperplane arrangement.  We have seen that
certain necessary properties are easy to deduce but it seems difficult to provide a
complete classification.  Do these ideals fit into some other well-known class?

\bibliographystyle{siam}
\bibliography{TropicalResolutions}

\end{document}

% Local Variables:
% mode: latex
% mode: font-lock
% mode: auto-fill
% fill-column: 90
% End: